\documentclass[reqno]{amsart}

\usepackage{tabu}
\usepackage{amssymb}
\usepackage{mathtools}
\usepackage{a4wide,amsmath}
\usepackage{mathrsfs}
\usepackage{amsthm}
\numberwithin{equation}{section}
\numberwithin{figure}{section}
\numberwithin{table}{section}
\usepackage{bbm}
\usepackage{subfig}
\usepackage{enumerate}
\usepackage[section]{placeins}
\usepackage{graphicx}		  
\usepackage{ifpdf}
\ifpdf
	\DeclareGraphicsExtensions{.pdf,.eps,.jpg,.png}	
	\usepackage[suffix=]{epstopdf}
\fi
\usepackage{xcolor}
\usepackage[utf8]{inputenc}
\usepackage{hyperref}
\hypersetup{hidelinks}

\long\def\MSC#1\EndMSC{\def\arg{#1}\ifx\arg\empty\relax\else
     {\narrower\noindent%
{2010 Mathematics Subject Classification}: #1\\} \fi}
\long\def\PACS#1\EndPACS{\def\arg{#1}\ifx\arg\empty\relax\else
     {\narrower\noindent%
{PACS numbers}: #1}\fi}
\long\def\KEY#1\EndKEY{\def\arg{#1}\ifx\arg\empty\relax\else
	{\narrower\noindent%
Keywords: #1\\}\fi}

\theoremstyle{plain}
\newtheorem{theorem}{Theorem}[section]
\newtheorem{lemma}[theorem]{Lemma}
\newtheorem{proposition}[theorem]{Proposition}
\newtheorem{corollary}[theorem]{Corollary}
\theoremstyle{definition}
\newtheorem{definition}[theorem]{Definition}

\theoremstyle{remark}
\newtheorem{remark}[theorem]{Remark}

\newcommand{\norm}[1]{\lVert#1\rVert}
\newcommand{\abs}[1]{\lvert#1\rvert} 
\newcommand{\inner}[1]{\langle#1\rangle}

\newcommand{\redel}{\mathop{\textup{Re}}}

\newcommand{\mspan}{\mathop{\textup{span}}}
\newcommand{\dist}{\mathop{\textup{dist}}}

\newcommand{\ident}{\mathop{\textup{id}}}

\newcommand{\rum}[1]{\mathbb{#1}}

\newcommand{\spanm}{\mathop{\textup{span}}}
\newcommand{\brum}{\mathbb{B}}

\newcommand{\as}{\hat{a}}
\newcommand{\ia}{\mathcal{I}_a}
\newcommand{\T}{\mathcal{T}}
\newcommand{\D}{\mathcal{D}}
\newcommand{\ga}{{g_a}}
\newcommand{\Ga}{{G_a}}
\newcommand{\Ka}{{K_a}}
\newcommand{\transp}[1]{{#1}^{\textup{T}}}
\newcommand{\diag}{\mathop{\textup{diag}}}
\newcommand{\Ccinf}{C_{\textup{c}}^\infty}
\newcommand{\I}{\mathrm{i}}    
\newcommand{\e}{\mathrm{e}}    
\newcommand{\di}{\mathrm{d}}   
\newcommand{\R}{\mathbb{R}}

\begin{document}

\title[Optimal distinguishability bounds for EIT]{Optimal depth-dependent distinguishability bounds for electrical impedance tomography in arbitrary dimension}

\author[H.~Garde]{Henrik Garde}
\address[H.~Garde]{Department of Mathematical Sciences, Aalborg University, Skjernvej 4A, 9220 Aalborg, Denmark.}
\email{henrik@math.aau.dk}

\author[N.~Hyv\"onen]{Nuutti Hyv\"onen}
\address[N.~Hyv\"onen]{Department of Mathematics and Systems Analysis, Aalto University, P.O. Box~11100, 02150 Espoo, Finland.}
\email{nuutti.hyvonen@aalto.fi}

\begin{abstract}
  The inverse problem of electrical impedance tomography is severely ill-posed. In particular, the resolution of images produced by impedance tomography deteriorates as the distance from the measurement boundary increases. Such depth dependence can be quantified by the concept of distinguishability of inclusions. This paper considers the distinguishability of perfectly conducting ball inclusions inside a unit ball domain, extending and improving known two-dimensional results to an arbitrary dimension $d\geq2$ with the help of Kelvin transformations. The obtained depth-dependent distinguishability bounds are also proven to be optimal. 
\end{abstract}

\maketitle

\KEY
electrical impedance tomography, 
Kelvin transformation,
depth dependence,
distinguishability.
\EndKEY

\MSC
35P15,
35R30,
47A30. 
\EndMSC

\section{Introduction}

The inverse conductivity problem in {\em Electrical Impedance Tomography} (EIT) is known to be highly ill-posed, and under reasonable assumptions it only allows conditional log-type stability estimates~\cite{Alessandrini1988,Mandache2001}. Such general estimates are uniform over the examined domain, although Lipschitz stability has been observed at the measurement boundary~\cite{Alessandrini1988,Alessandrini2001,Alessandrini2009,Brown2001a,Sylvester1988,Kang2002,Nakamura2001a,Nakamura2001,Nakamura2003}. This suggests the stability of the inverse conductivity problem is actually  depth-dependent, which is also reflected in the quality of numerical reconstructions. Using the notion of \emph{distinguishability}, some results characterizing this depth dependence have recently been obtained in two dimensions~\cite{Alessandrini_2017,Garde_2017}. Motivated by the inherent three-dimensionality of EIT, this work extends and improves the depth-dependent distinguishability bounds for perfectly conducting inclusions presented in \cite{Garde_2017} to an arbitrary spatial dimension $d\geq 2$. Although we focus here solely on the nonlinear inverse conductivity problem, it should be acknowledged that there also exist previous works tackling depth-dependent sensitivity for its linearized version~\cite{Ammari_2013,Nagayasu_2009}.

The main tool in our analysis is the Kelvin transformation~\cite{Kelvin} that takes the role played by M\"obius transformations in~\cite{Garde_2017}. The Kelvin transformation is a traditional tool in,~e.g.,~potential theory for problems in unbounded domains, and it is typically defined using the inversion in the unit sphere \cite{Axler_2001,Bogdan_2006,Michalik_2012,Wermer_1974}. However, we need to consider inversions with respect to arbitrary spheres in our analysis, leading to the employment of translated and dilated versions of the classic Kelvin transformation (cf.~\cite{Armitage_2001,Hanke_2011}). A central property of all Kelvin transformations, relating them to EIT and also explaining their use in potential theory, is that a function is harmonic if and only if its Kelvin transformation is harmonic. Compared to the use of M\"obius transformations for $d=2$ in~\cite{Garde_2017}, a complication related to Kelvin transformations for $d > 2$ is their habit to mix Neumann and Robin boundary conditions. This currently restricts our analysis in higher dimensions to perfectly conducting inclusions, characterized by a homogeneous Dirichlet condition on the inclusion boundary.

Let $\brum$ denote the Euclidean unit ball in $\rum{R}^d$ for any integer $d\geq 2$. We consider the setting where a single \emph{perfectly conducting} concentric ball $B(0,r)$ of radius $0<r<1$ is mapped by an inversion, which leaves $\brum$ invariant, onto a nonconcentric ball $B(C,R) \subset \brum$ centered at $C \in \brum$. It turns out that for any ball $B(C,R)\subset \brum$ there exists a unique inversion, that relates it in this manner, to a concentric ball $B(0,r)$ with $r = r(\abs{C},R) \in(0,1)$. Alternatively, one can treat $C$ and $R$ as functions of $r\in(0,1)$ and a vector $a\in\brum\setminus\{0\}$ parametrizing all inversions mapping $\brum$ onto itself. To be more precise, $a \in B(C,R)$ is the image of the origin under the considered inversion, and thus $\rho := \abs{a}\in(0,1)$ can be interpreted as a parameter controlling the `nonconcentricity' or `depth' of $B(C,R)$. Indeed, a~small $\rho$ corresponds to an almost concentric inclusion, whereas $\rho$ close to $1$ indicates that $B(C,R)$ lies close to $\partial\brum$.

Let $\Lambda_{0,r}$ and $\Lambda_{C,R}$ denote the {\em Dirichlet-to-Neumann} (DN) maps on $\partial\brum$ when a perfectly conducting inclusion is placed on $B(0,r)$ and on $B(C,R)$, respectively. In both cases, the remainder of $\brum$ is characterized by unit conductivity, i.e.~by the Laplace equation. Furthermore, let $\Lambda_1$ be the DN map on $\partial \brum$ for the inclusion-free problem with unit conductivity. It is well known that  $\Lambda_{0,r}-\Lambda_1$ and $\Lambda_{C,R}-\Lambda_1$ are smoothening \cite{Lee1989} and belong in particular to $\mathscr{L}(L^2(\partial\brum))$, the space of bounded linear operators on $L^2(\partial\brum)$. The distinguishability of the inclusion $B(C,R)$ is defined to be $\norm{\Lambda_{C,R}-\Lambda_1}_{\mathscr{L}(L^2(\partial\brum))}$ (see,~e.g.,~\cite{Alessandrini_2017,Cheney1992,Garde_2017,Isaacson1986}), and it can be motivated as follows: Assume $\Lambda$ is a DN map on $\partial \brum$ corrupted by an additive self-adjoint noise perturbation $E\in \mathscr{L}(L^2(\partial\brum))$. If the size of this perturbation, $\norm{E}_{\mathscr{L}(L^2(\partial\brum))}$, is larger than the distinguishability of $B(C,R)$ and without further information on the structure of $E$, it is impossible to determine if the datum in hand corresponds to $B(C,R)$ embedded in $\brum$ or to an inclusion-free $\brum$.

The main result of this paper (Theorem~\ref{thm:L2bnds}) relates the distinguishabilities of $B(0,r)$ and $B(C,R)$ as follows:
\begin{equation}
	\frac{1-\rho}{1+\rho} \leq \frac{\norm{\Lambda_{0,r}-\Lambda_1}_{\mathscr{L}(L^2(\partial\brum))}}{\norm{\Lambda_{C,R}-\Lambda_1}_{\mathscr{L}(L^2(\partial\brum))}} \leq \frac{1-\rho^2}{1+\rho^2}. \label{eq:intro}
\end{equation}   
Observe that the ratio of the operator norms in \eqref{eq:intro} can be interpreted as a function of $r \in (0,1)$ even if $\rho = \abs{a}$ is given; fixing $\rho$ determines the employed Kelvin transformation, or inversion, up to a rotation of~$\brum$, but the sizes of the considered inclusions $B(0,r)$ and $B(C(a,r), R(\rho,r))$ are still controlled by $r \in (0,1)$.  The estimates in \eqref{eq:intro} are optimal in the sense that the supremum and the infimum of the norm ratio over $r\in(0,1)$ exactly give the upper and lower bounds in \eqref{eq:intro}, respectively. To be more precise, the lower bound is reached when $r \to 1^-$ and the upper bound when $r \to 0^+$.

Both the upper and the lower bound in \eqref{eq:intro} converge to zero as $\rho \to 1^-$, which characterizes how a perfectly conducting inclusion $B(C, R)$ becomes more distinguishable to EIT when it approaches the measurement boundary --- even though $R = R(\rho,r)$ converges to zero when $\rho \to 1^-$ for any fixed $r \in (0,1)$.  To the authors' knowledge, \eqref{eq:intro} provides the first result of this kind for $d>2$. In particular, note that the bounds in \eqref{eq:intro} are independent of the dimension $d$, as are the formulas for $C$ and $R$ as functions of $a$ and $r$. Consult Theorem~\ref{thm:balls} in Section~\ref{sec:kelvinball} for the explicit relation between $a$, $r$, $C$, and $R$. 

The Riemann mapping theorem of harmonic morphisms does not have a counterpart for $d > 2$, which partially explains the simple geometric setting of our paper, albeit the two-dimensional analysis of \cite{Garde_2017} also only considered discoidal inclusions inside the unit disk. Be that as it may, we still expect our results can shed some light on the distinguishability of inclusions in more complicated geometric setups as well.

This paper is organized as follows. In Section~\ref{sec:kelvin}, we introduce the general Kelvin transformations on distributions and remind the reader about a commutation relation between the Kelvin transformation and the Laplacian (Theorem~\ref{thm:kelvinlaplace}). Section~\ref{sec:kelvinball} focuses on the unit ball $\brum$ and gives a characterization of an inversion that maps a given ball inside $\brum$ onto a ball concentric with $\brum$ 
while leaving $\brum$ itself invariant (Theorem~\ref{thm:balls}). The behavior of the associated Kelvin transformations on $\partial \brum$ is also investigated. In Section~\ref{sec:EIT}, we relate the DN maps corresponding to concentric and nonconcentric perfectly conducting inclusions via Kelvin transformations (Theorem~\ref{thm:DNkelvin}). Finally, Section~\ref{sec:bounds} proves our main result (Theorem~\ref{thm:L2bnds}). The paper is completed by two appendices. Appendix~\ref{sec:appA} analyzes the relation between Kelvin and M\"obius transformations in two dimensions, in order to ease the comparison of our results and techniques with those in \cite{Garde_2017}. Appendix~\ref{sec:appB} gives a representation formula for the Kelvin transformed DN map, which may be useful for numerical simulations.

\subsection{Notational remarks}

We denote by $\mathscr{L}(X,Y)$ the space of bounded linear operators between Banach spaces $X$ and $Y$, and introduce the shorthand notation $\mathscr{L}(X) := \mathscr{L}(X,X)$. 

For any scalar-valued function $g$, we denote by $g^k$, $k \in \mathbb{Z}$, the composition of the $k$'th power and~$g$. That is, $g^k(x) := g(x)^k$ for any $k \in \mathbb{Z}$ and all $x$ in the domain of $g$ for which $g(x)^k$ is defined. In particular, $g^{-1} = 1/g$ is {\em not} the inverse of $g$. The `power notation' has its usual meaning for linear operators and matrices. In particular, $A^{-1}$ is the inverse of the matrix/operator $A$.

The open Euclidean ball and sphere  with center $\mathcal{C}\in\rum{R}^d$ and radius $\mathcal{R}>0$ are denoted by $B(\mathcal{C},\mathcal{R})$ and $S(\mathcal{C},\mathcal{R}) := \partial B(\mathcal{C},\mathcal{R})$, respectively. We use the shorthand notation $\brum := B(0,1)$ for the unit ball and denote by $\di S$ the standard (unnormalized) spherical measure on $\partial\brum$.

The Euclidean norm is denoted by $\abs{ \, \cdot \, }$, and $\{e_j\}_{j=1}^d$ is the canonical basis for $\rum{R}^d$. For $a\in\rum{R}^d\setminus\{0\}$, we set $e_a := \smash{\frac{a}{\abs{a}}}\in \partial\brum$.

\section{Kelvin transformations} \label{sec:kelvin}

The {\em inversion} in the sphere $S(\mathcal{C},\mathcal{R})$, denoted $I_{\mathcal{C},\mathcal{R}} : \rum{R}^d\setminus\{\mathcal{C} \} \to \rum{R}^d\setminus\{\mathcal{C}\}$, is defined as
\begin{equation*}
I_{\mathcal{C},\mathcal{R}}(x) := \mathcal{R}^2\frac{x-\mathcal{C}}{\abs{x-\mathcal{C}}^2} + \mathcal{C} = g_{\mathcal{C},\mathcal{R}}^2(x)(x-\mathcal{C}) + \mathcal{C},
\end{equation*}
where $g_{\mathcal{C},\mathcal{R}}(x) := \mathcal{R}/\abs{x-\mathcal{C}}$.  Note that $I_{\mathcal{C},\mathcal{R}}$ maps the points in the interior of $S(\mathcal{C},\mathcal{R})$ to its exterior and vice versa. Furthermore, $I_{\mathcal{C},\mathcal{R}}(x)$ is the unique point on the half-line from $\mathcal{C}$ through $x$ satisfying
\begin{equation}
\abs{I_{\mathcal{C},\mathcal{R}}(x)-\mathcal{C}}\abs{x-\mathcal{C}} = \mathcal{R}^2 \label{eq:eqident}
\end{equation}
for any $x\not=\mathcal{C}$.
This clearly indicates the fixed points of $I_{\mathcal{C},\mathcal{R}}$ are precisely $S(\mathcal{C},\mathcal{R})$. Moreover, $I_{\mathcal{C},\mathcal{R}}\circ I_{\mathcal{C},\mathcal{R}} = \ident$, i.e.,~$I_{\mathcal{C},\mathcal{R}}$ is an involution. The inversion $I_{\mathcal{C},\mathcal{R}}$ maps spheres and hyperplanes onto spheres and hyperplanes, but not necessarily respectively and with possibly $\mathcal{C}$ removed~\cite{Armitage_2001}. In particular, as $I_{\mathcal{C},\mathcal{R}}$ clearly maps any bounded set $E \subset \rum{R}^d$ with $\dist(E,\mathcal{C})>0$ to a bounded set, it sends spheres not intersecting $\mathcal{C}$ to spheres. We use the special notation $\hat{x} := I_{0,1}(x)$, $x \in \R^d/\{0\}$, for the inversion in the unit sphere $\partial\brum = S(0,1)$.

Let $\mathcal{C} \in \R^d$, choose any open $\Omega\subseteq \rum{R}^d\setminus\{\mathcal{C}\}$, and set  $\Omega^* := I_{\mathcal{C},\mathcal{R}}(\Omega)$. The composition of a function with $I_{\mathcal{C},\mathcal{R}}$ can be interpreted as a linear operator from $L^1_{\rm loc}(\Omega)$ to $L^1_{\rm loc}(\Omega^*)$ or vice versa,~i.e.,~we define
\begin{equation}
  \label{eq:Inversion}
\mathcal{I}_{\mathcal{C},\mathcal{R}} f := f \circ I_{\mathcal{C},\mathcal{R}} 
\end{equation}
for all $f \in L^1_{\rm loc}(\Omega)$ or $f \in L^1_{\rm loc}(\Omega^*)$.
In particular,  $\mathcal{I}_{\mathcal{C},\mathcal{R}}$ is an involution. We also introduce a linear multiplication operator $G_{\mathcal{C},\mathcal{R}}$ via
\begin{equation*}
	G_{\mathcal{C},\mathcal{R}} f := g_{\mathcal{C},\mathcal{R}} f.
\end{equation*}
Since $\mathcal{C} \notin \Omega$, it is obvious that $G_{\mathcal{C},\mathcal{R}}$ maps $L^1_{\rm loc}(\Omega)$ onto itself, and its inverse $G_{\mathcal{C},\mathcal{R}}^{-1}: L^1_{\rm loc}(\Omega) \to L^1_{\rm loc}(\Omega)$ is defined as the multiplication by $g_{\mathcal{C},\mathcal{R}}^{-1}$. The same conclusions also apply on $L^1_{\rm loc}(\Omega^*)$ because $\mathcal{C} \notin \Omega^*$ as well.
A straightforward calculation leads to the identities
\begin{equation}
  \label{eq:IGcommut}
  G_{\mathcal{C},\mathcal{R}} \mathcal{I}_{\mathcal{C},\mathcal{R}} = \mathcal{I}_{\mathcal{C},\mathcal{R}} G_{\mathcal{C},\mathcal{R}}^{-1}  \qquad {\rm and} \qquad
  G_{\mathcal{C},\mathcal{R}}^{-1} \mathcal{I}_{\mathcal{C},\mathcal{R}} = \mathcal{I}_{\mathcal{C},\mathcal{R}} G_{\mathcal{C},\mathcal{R}},
  \end{equation}
which will be frequently utilized in what follows.

\begin{definition} \label{defi:kelvin}
	The general \emph{Kelvin transformation} is defined as a linear map $K_{\mathcal{C},\mathcal{R}}: L^1_{\rm loc}(\Omega) \to L^1_{\rm loc}(\Omega^*)$, or $K_{\mathcal{C},\mathcal{R}}: L^1_{\rm loc}(\Omega^*) \to L^1_{\rm loc}(\Omega)$, by setting
	\begin{equation}
	\label{eq:Kelvin}
	K_{\mathcal{C},\mathcal{R}} := G_{\mathcal{C},\mathcal{R}}^{d-2}\mathcal{I}_{\mathcal{C},\mathcal{R}} = \mathcal{I}_{\mathcal{C},\mathcal{R}}G^{2-d}_{\mathcal{C},\mathcal{R}}.
	\end{equation}
\end{definition}
\begin{remark}
	The latter equality in Definition~\ref{defi:kelvin} follows from \eqref{eq:IGcommut} and demonstrates that $K_{\mathcal{C},\mathcal{R}}$ is also an involution. Moreover, the Kelvin transformation obviously inherits the identities 
	\begin{equation}
	\label{eq:KGcommut}
	G_{\mathcal{C},\mathcal{R}} K_{\mathcal{C},\mathcal{R}} = K_{\mathcal{C},\mathcal{R}} G_{\mathcal{C},\mathcal{R}}^{-1}  \qquad {\rm and} \qquad
	G_{\mathcal{C},\mathcal{R}}^{-1} K_{\mathcal{C},\mathcal{R}} = K_{\mathcal{C},\mathcal{R}} G_{\mathcal{C},\mathcal{R}}
	\end{equation}
	from the associated inversion operator, cf.~\eqref{eq:IGcommut}.
\end{remark}
As preparation for the following result: for any $y\in\rum{R}^d\setminus\{0\}$, let $P_y$ be the orthogonal projection onto the line $\spanm\{y\}$, and let $Q_y \in \smash{\R^{d \times d}}$ be an arbitrary matrix whose first row is $\smash{\transp{y}/\abs{y}}$ and the remaining $d-1$ rows form an orthonormal basis for $\smash{\spanm\{y\}^\perp}$. Clearly $Q_y$ is orthogonal and satisfies $Q_y y = \abs{y}e_1$.
\begin{proposition} \label{prop:generalJac}
	The Jacobian matrix $J_{\mathcal{C},\mathcal{R}}$ of $I_{\mathcal{C},\mathcal{R}}$ on $\rum{R}^d\setminus\{\mathcal{C}\}$ is
	\begin{align}
		J_{\mathcal{C},\mathcal{R}}(x) &= g_{\mathcal{C},\mathcal{R}}^2(x)(\ident - 2P_{x-\mathcal{C}}) \label{eq:J1} \\
		&= g_{\mathcal{C},\mathcal{R}}^2(x) Q_{x-\mathcal{C}}^{-1}\diag(-1,1,1,\dots,1)Q_{x-\mathcal{C}}, \label{eq:J2}
	\end{align} 
	As direct consequences,
	\begin{enumerate}[(i)]
	\item $\det(J_{\mathcal{C},\mathcal{R}}) = -g_{\mathcal{C},\mathcal{R}}^{2d}$,
          \vspace{1mm}
		\item $\transp{J}_{\mathcal{C},\mathcal{R}} = J_{\mathcal{C},\mathcal{R}}$, $J_{\mathcal{C},\mathcal{R}}^2 = g_{\mathcal{C},\mathcal{R}}^4\ident$, and $J_{\mathcal{C},\mathcal{R}}^{-1} = g_{\mathcal{C},\mathcal{R}}^{-4}J_{\mathcal{C},\mathcal{R}}$.
	\end{enumerate}
\end{proposition}
\begin{proof}
	The expression \eqref{eq:J1} follows directly from 
	\begin{equation*}
	\frac{\partial}{\partial x_j}I_{\mathcal{C},\mathcal{R}}(x) = \mathcal{R}^2\left(\frac{1}{\abs{x-\mathcal{C}}^2}e_j - 2\frac{(x-\mathcal{C})_j}{\abs{x-\mathcal{C}}^4}(x-\mathcal{C}) \right), \qquad x\in\rum{R}^d\setminus\{\mathcal{C}\}, 
	\end{equation*}
	which readily yields
	\begin{equation*}
	J_{\mathcal{C},\mathcal{R}}(x)y = g_{\mathcal{C},\mathcal{R}}^2(x)\left(y-2\frac{y\cdot(x-\mathcal{C})}{\abs{x-\mathcal{C}}^2}(x-\mathcal{C})\right) = g_{\mathcal{C},\mathcal{R}}^2(x)(\ident - 2P_{x-\mathcal{C}})y
	\end{equation*}
        for any $y\in\rum{R}^d$ and $x\in \R^d / \{\mathcal{C}\}$.

	As $Q_{x-\mathcal{C}}$ rotates the coordinate system so the direction of $x-\mathcal{C}$ corresponds to the first coordinate, we have
    \begin{equation*}
    	P_{x-C} = Q_{x-\mathcal{C}}^{-1} \diag(1,0,0,\dots,0) Q_{x-\mathcal{C}}.
    \end{equation*}
    Hence,
    \begin{equation*}
    	J_{\mathcal{C},\mathcal{R}}(x) = g_{\mathcal{C},\mathcal{R}}^2(x)(\ident - 2P_{x-\mathcal{C}}) = g_{\mathcal{C},\mathcal{R}}^2(x) Q_{x-\mathcal{C}}^{-1} \big(\ident -  \diag(2,0,0,\dots,0) \big)Q_{x-\mathcal{C}}, \qquad x \in \R^d / \{\mathcal{C} \},
    \end{equation*}
    which proves \eqref{eq:J2}. Finally, (i) and (ii) immediately follow from \eqref{eq:J2} since $Q_{x-\mathcal{C}}$ is orthogonal.
\end{proof}

Obviously, $I_{\mathcal{C},\mathcal{R}}$ maps compact subsets of $\Omega$ to compact subsets of $\Omega^*$ and vice versa. Since $I_{\mathcal{C},\mathcal{R}}$ is smooth on $\rum{R}^d\setminus\{\mathcal{C}\}$,  it holds that $\mathcal{I}_{\mathcal{C},\mathcal{R}}(\Ccinf(\Omega)) \subseteq \Ccinf(\Omega^*)$ and $\mathcal{I}_{\mathcal{C},\mathcal{R}}(\Ccinf(\Omega^*)) \subseteq \Ccinf(\Omega)$. By applying $\mathcal{I}_{\mathcal{C},\mathcal{R}}$ to these inclusions, the reversed ones $\Ccinf(\Omega) \subseteq \mathcal{I}_{\mathcal{C},\mathcal{R}}(\Ccinf(\Omega^*))$ and $\Ccinf(\Omega^*)\subseteq \mathcal{I}_{\mathcal{C},\mathcal{R}}(\Ccinf(\Omega))$ follow. Since $g_{\mathcal{C},\mathcal{R}}$ is smooth and positive on $\rum{R}^d\setminus\{\mathcal{C}\}$, all these inclusions also hold when $\mathcal{I}_{\mathcal{C},\mathcal{R}}$ is replaced by $G_{\mathcal{C},\mathcal{R}}^k \mathcal{I}_{\mathcal{C},\mathcal{R}}$ for any $k \in \mathbb{Z}$. Altogether, we have thus established
\begin{equation*}
	G_{\mathcal{C},\mathcal{R}}^k \mathcal{I}_{\mathcal{C},\mathcal{R}}(\Ccinf(\Omega)) = \Ccinf(\Omega^*) = K_{\mathcal{C},\mathcal{R}}(\Ccinf(\Omega)), \qquad k \in \mathbb{Z},
\end{equation*}
which naturally also holds if the roles of $\Omega$ and $\Omega^*$ are reversed.

In consequence, we may extend both $K_{\mathcal{C},\mathcal{R}}$ and $\mathcal{I}_{\mathcal{C},\mathcal{R}}$ to continuous linear maps on distributions from $\D'(\Omega)$ to $\D'(\Omega^*)$ by setting
\begin{align}
 \label{eq:kelvindist}
 \inner{K_{\mathcal{C},\mathcal{R}}u, \phi} &:= \inner{u,G_{\mathcal{C},\mathcal{R}}^4 K_{\mathcal{C},\mathcal{R}}\phi}, \\[1mm]
 \inner{\mathcal{I}_{\mathcal{C},\mathcal{R}}u, \phi} &:= \inner{u,G_{\mathcal{C},\mathcal{R}}^{2d} \mathcal{I}_{\mathcal{C},\mathcal{R}}\phi}, \label{eq:invdist}
\end{align}
for any $u \in \D'(\Omega)$, all $\phi \in \Ccinf(\Omega^*)$, and with $\inner{\cdot,\cdot}$ denoting the dual pairing of $(\D',\Ccinf)$. The definition \eqref{eq:kelvindist} coincides with that of the {\em distributional Kelvin transformation} in~\cite{Hanke_2011}, and it can be motivated as follows: for any $u \in L^1_{\rm loc}(\Omega)$ and $\phi \in \Ccinf(\Omega^*)$, we have
\begin{align}
  \label{eq:kelvindist_mot}
\int_{\Omega^*} K_{\mathcal{C},\mathcal{R}} u \, \phi \, {\rm d} x &=
\int_{\Omega^*} \mathcal{I}_{\mathcal{C},\mathcal{R}} u \, G_{\mathcal{C},\mathcal{R}}^{d-2} \phi \, \di x = \int_{\Omega} u \, (\mathcal{I}_{\mathcal{C},\mathcal{R}} G_{\mathcal{C},\mathcal{R}}^{d-2} \phi) \,\abs{\det(J_{\mathcal{C}, \mathcal{R}})} \, \di x \nonumber \\
& =  \int_{\Omega} u \, G^{2d}_{\mathcal{C},\mathcal{R}}\mathcal{I}_{\mathcal{C},\mathcal{R}} G^{d-2}_{\mathcal{C},\mathcal{R}} \phi  \, \di x
=  \int_{\Omega} u \, G_{\mathcal{C},\mathcal{R}}^{4} K_{\mathcal{C},\mathcal{R}} \phi   \, \di x,
\end{align}
where we used Proposition~\ref{prop:generalJac}(i), \eqref{eq:IGcommut}, and \eqref{eq:Kelvin}. In other words, \eqref{eq:kelvindist} coincides with the standard Kelvin transformation \eqref{eq:Kelvin} on $L^1_{\rm loc}(\Omega)$. The same conclusion also applies to \eqref{eq:invdist}. It is easy to check that $K_{\mathcal{C}, \mathcal{R}}: \D'(\Omega) \to \D'(\Omega^*)$ and $\mathcal{I}_{\mathcal{C}, \mathcal{R}}: \D'(\Omega) \to \D'(\Omega^*)$ are involutions satisfying \eqref{eq:KGcommut} and \eqref{eq:IGcommut}, respectively.

If $\Omega$ is also bounded and satisfies $\dist(\Omega,\mathcal{C})>0$, it holds 
\begin{equation*}
G_{\mathcal{C},\mathcal{R}}^k \mathcal{I}_{\mathcal{C},\mathcal{R}}(H^m(\Omega)) = H^m(\Omega^*) = K_{\mathcal{C},\mathcal{R}}(H^m(\Omega)), \qquad k \in \mathbb{Z},
\end{equation*}
for any $m\in\rum{Z}$ and also with the roles of $\Omega$ and $\Omega^*$ reversed. To see this, note first that $\Omega$ is bounded if and only if $\dist(\Omega^*,\mathcal{C})>0$. Hence, the derivatives of $I_{\mathcal{C},\mathcal{R}}$ and $g_{\mathcal{C},\mathcal{R}}$, up to an arbitrary order $\abs{m}$, are uniformly bounded by a constant (depending on $\abs{m}$) in both $\Omega$ and $\Omega^*$. Therefore, for any $k \in \mathbb{Z}$, $G_{\mathcal{C},\mathcal{R}}^k \mathcal{I}_{\mathcal{C},\mathcal{R}}(H^m(\Omega)) \subseteq H^m(\Omega^*)$ and $G_{\mathcal{C},\mathcal{R}}^k \mathcal{I}_{\mathcal{C},\mathcal{R}}(H^m(\Omega^*)) \subseteq H^m(\Omega)$. Applying $G_{\mathcal{C},\mathcal{R}}^k \mathcal{I}_{\mathcal{C},\mathcal{R}}$ to the latter inclusion yields $H^m(\Omega^*) \subseteq G_{\mathcal{C},\mathcal{R}}^k \mathcal{I}_{\mathcal{C},\mathcal{R}}(H^m(\Omega))$, which proves the claim. In fact,
\begin{equation}
  \label{eq:Kelvin_norm}
\norm{G_{\mathcal{C},\mathcal{R}}^k \mathcal{I}_{\mathcal{C},\mathcal{R}} u}_{H^m(\Omega^*)} \leq C(\Omega,\mathcal{R},\mathcal{C},k,m) \norm{u}_{H^m(\Omega)} \qquad \text{for all } u \in H^m(\Omega),
\end{equation}
demonstrating, in particular, the boundedness of $K_{\mathcal{C},\mathcal{R}}: H^m(\Omega) \to H^m(\Omega^*)$. The same conclusion naturally applies to the inverse $K_{\mathcal{C},\mathcal{R}}: H^m(\Omega^*) \to H^m(\Omega)$ as well.

\begin{definition} \label{defi:translationAndDilation}
	The usual translation and dilation operators are defined on locally integrable functions by $\T_af(x) := f(x-a)$ and $\D_bf(x) := f(b^{-1}x)$ for $a\in\rum{R}^d$ and $b>0$.	
\end{definition}

Via a direct calculation, we see that $K_{\mathcal{C},\mathcal{R}}: L^1_{\rm loc}(\Omega) \to L^1_{\rm loc}(\Omega^*)$ satisfies
\begin{equation}
K_{\mathcal{C},\mathcal{R}} = \T_\mathcal{C} \D_\mathcal{R} K_{0,1} \D_{\mathcal{R}^{-1}} \T_{-\mathcal{C}}, \label{eq:kelvintransdilate}
\end{equation}
and thus all Kelvin transformations are dilated and translated variants of $K_{0,1}$. Take note that \eqref{eq:kelvintransdilate} also holds for the distributional Kelvin transformation, if the the dilation and translation are defined on distributions via the dual pairing in the natural manner, i.e., in the way that the extended definitions coincide with the original ones for locally integrable functions.

We now arrive at the following fundamental theorem.

\begin{theorem} \label{thm:kelvinlaplace}
	The commutation relation 
	\begin{equation*}
	\Delta K_{\mathcal{C},\mathcal{R}} = G_{\mathcal{C},\mathcal{R}}^4 K_{\mathcal{C},\mathcal{R}}\Delta
	\end{equation*}
holds on $\D'(\rum{R}^d\setminus\{\mathcal{C}\})$.
\end{theorem}
\begin{proof}
 	The result is well known to hold for $K_{0,1}$ on $C^2(\rum{R}^d\setminus\{0\})$ (see.,~e.g.,~\cite[Theorem 1.6.3]{Armitage_2001} or \cite[Proposition 4.6]{Axler_2001}), that is,
	\begin{equation*}
	\Delta K_{0,1}u = G_{0,1}^4 K_{0,1}\Delta u
	\end{equation*}
	for all $u\in C^2(\rum{R}^d\setminus\{0\})$.
	Employing \eqref{eq:kelvintransdilate} and the identity $G_{\mathcal{C},\mathcal{R}}\T_\mathcal{C}\D_\mathcal{R} = \T_\mathcal{C}\D_\mathcal{R} G_{0,1}$, we obtain
	\begin{align*}
	\Delta K_{\mathcal{C},\mathcal{R}} &= \Delta\T_\mathcal{C} \D_\mathcal{R} K_{0,1} \D_{\mathcal{R}^{-1}} \T_{-\mathcal{C}} \\
	&= \mathcal{R}^{-2}\T_\mathcal{C} \D_\mathcal{R} \Delta K_{0,1} \D_{\mathcal{R}^{-1}} \T_{-\mathcal{C}} \\
	&= \mathcal{R}^{-2}\T_\mathcal{C} \D_\mathcal{R} G_{0,1}^4 K_{0,1} \Delta \D_{\mathcal{R}^{-1}} \T_{-\mathcal{C}} \\
	&= G_{\mathcal{C},\mathcal{R}}^4\T_\mathcal{C} \D_\mathcal{R} K_{0,1} \D_{\mathcal{R}^{-1}} \T_{-\mathcal{C}} \Delta \\
	&= G_{\mathcal{C},\mathcal{R}}^4 K_{\mathcal{C},\mathcal{R}}\Delta
	\end{align*}
	on $C^2(\rum{R}^d\setminus\{\mathcal{C}\})$. To finalize the proof, let $u\in \D'(\rum{R}^d\setminus\{\mathcal{C}\})$ and $\phi\in \Ccinf(\rum{R}^d\setminus\{\mathcal{C}\})$ be arbitrary. We apply the definition of distributional differentiation and \eqref{eq:kelvindist} to deduce
	\begin{align*}
	\inner{\Delta K_{\mathcal{C},\mathcal{R}}u,\phi} &= \inner{u,G_{\mathcal{C},\mathcal{R}}^4K_{\mathcal{C},\mathcal{R}}\Delta\phi} \\
	&= \inner{u,\Delta K_{\mathcal{C},\mathcal{R}}\phi} \\
	&= \inner{\Delta u,G_{\mathcal{C},\mathcal{R}}^4 K_{\mathcal{C},\mathcal{R}}G_{\mathcal{C},\mathcal{R}}^4 \phi} \\
	&= \inner{G_{\mathcal{C},\mathcal{R}}^4 K_{\mathcal{C},\mathcal{R}}\Delta u,\phi},
	\end{align*}
    where we also used \eqref{eq:KGcommut} in the third equality.
\end{proof}

\subsection{Kelvin transformations on the unit ball} \label{sec:kelvinball}
From now on, we restrict our attention to the unit ball,~i.e.~choose $\Omega=\brum$, and concentrate on such inversions that also $\Omega^* = \brum$. However, it would be straightforward to generalize our results to a ball of any radius. We will employ the symbols $C,R$, instead of $\mathcal{C},\mathcal{R}$, to annotate a ball $B(C,R)$ embedded in $\brum$. 

Let $a\in\brum\setminus\{0\}$ and write $a = \rho e_a$ with $\rho = \abs{a}$ and $e_a\in\partial\brum$. Recall the notation $\as = I_{0,1}(a)$, notice that $\abs{\as} = \smash{\rho^{-1}} > 1$, and define $b := \smash{(\rho^{-2}-1)^{1/2}} = \smash{\rho^{-1}(1-\rho^2)^{1/2}}$. We introduce a special Kelvin transformation on $\smash{L^1_{\rm loc}(\brum)}$ by setting $\Ka := K_{\as, b}$; similarly, we also define $\ga := g_{\as,b}$, $\Ga := G_{\hat{a},b}$, $I_a := I_{\as,b}$, and $\ia := \mathcal{I}_{\hat{a},b}$. More explicitly, 
\begin{equation*}
\Ka f(x) = G_a^{d-2}\ia f(x) = \left( \frac{b}{\abs{x-\as}} \right)^{d-2}\!\!f\left( b^2\frac{x-\as}{\abs{x-\as}^2} + \as \right), \qquad f \in L^1_{\rm loc}(\brum),
\end{equation*}
and this definition naturally extends to $\mathcal{D}'(\brum)$ through \eqref{eq:kelvindist}.

The following theorem shows that $I_a$ leaves $\brum$ invariant for any $a \in \brum\setminus\{0\}$, and it also gives a characterization of how any nonconcentric ball inside $\brum$ can be mapped to a concentric one by~$I_a$ with a suitable $a$. It is worth noting that the formulas \eqref{eq:CR} and \eqref{eq:ar} below generalize the two-dimensional result in \cite[Proposition~2.1]{Garde_2017}; see Appendix~\ref{sec:appA}. Moreover, an equivalent characterization for three spatial dimensions can be found in~\cite{Hanke_2011}.

\begin{theorem} \label{thm:balls}
Assume $a\in\brum\setminus\{0\}$. The inversion $I_a$ leaves $\overline{\brum}$ invariant, that is, $I_a(\brum) = \brum$ and $I_a(\partial\brum) = \partial\brum$. In particular,
\begin{equation*}
	I_a|_{\partial\brum}(x) = (\ident - 2P_{x-\as})x, \qquad x \in \partial\brum,
\end{equation*}
with $P_y$ denoting the orthogonal projection onto the line spanned by $y\in\rum{R}^d\setminus\{0\}$.

The following two items completely characterize how a concentric ball embedded in $\brum$ is deformed under a given $I_a$, as well as which $I_a$ maps a given ball embedded in $\brum$ to a concentric one:  
	\begin{enumerate}[(i)]
		\item Let $r\in(0,1)$ and $a = \rho e_a$ with $\rho\in(0,1)$ and $e_a\in\partial\brum$. Then $I_a(B(0,r)) = B(C,R)$ and $I_a(S(0,r)) = S(C,R)$ with
		\begin{equation}
		C = \frac{\rho(r^2-1)}{\rho^2r^2-1}e_a, \qquad R = \frac{r(\rho^2-1)}{\rho^2r^2-1}. \label{eq:CR}
		\end{equation} 	
		\item Let $R\in(0,1-c)$ and $C = ce_a$ with $c\in(0,1)$ and $e_a\in\partial\brum$. Then $I_a(B(C,R)) = B(0,r)$ and $I_a(S(C,R)) = S(0,r)$ with
		\begin{equation}
		r = \frac{1+R^2-c^2-\sqrt{((1-R)^2-c^2)((1+R)^2-c^2)}}{2R}, \qquad a = \frac{C}{1-Rr}. \label{eq:ar}
		\end{equation}
	\end{enumerate}	
\end{theorem}
\begin{proof}
	Assuming $\abs{x}=1$, we obtain  
	\begin{equation*}
	  I_a(x) = \frac{(\abs{\as}^2-1)(x-\as)}{\abs{x-\as}^2}+\as = \frac{(\abs{\as}^2-1)x+2(1-x\cdot\as)\as}{\abs{x-\as}^2},
        \end{equation*}
        as well as
        \begin{equation*}
	(\ident -2P_{x-\as})x = x-2\frac{x\cdot(x-\as)}{\abs{x-\as}^2}(x-\as) = \frac{(\abs{\as}^2-1)x+2(1-x\cdot\as)\as}{\abs{x-\as}^2}.
	\end{equation*}
        These formulas verify the claimed representation for $I_a$ on $\partial \brum$.
	
	Let us then prove that $I_a$ maps the closure of $\brum$ onto itself. Clearly, $I_{a}$ sends the points on the line spanned by $a$ to that same line, with the exception of $\as \notin \brum $ that is mapped to infinity and is the only point of discontinuity for $I_{a}$. Let $\brum_a := \{t e_a \mid t \in(-1,1)\}$. As $I_{a}(\pm e_a) = \mp e_a$ and $I_{a}(0) = a\in\brum_a$, it follows from the continuity of $I_{a}$ that $I_{a}(\overline{\brum_a}) = \overline{\brum_a}$. Since $I_{a}$ is the inversion in the sphere $S(\hat{a},b)$,  it is symmetric about the line spanned by $a$. In particular, $I_{a}$ maps any sphere centered on the line spanned by $a$, and not intersecting $\as$, onto another sphere centered on that very same line. Hence, $I_{a}(\partial\brum) = \partial\brum$ because $I_{a}(\partial\brum)$ is known to contain $\pm e_a$. As $I_a$ is a continuous involution away from $\as \notin \brum$ and $I_a(0) \in \brum$, it must in fact hold $I_{a}(\brum) = \brum$.
	
	Because \eqref{eq:ar} in part (ii) of the assertion follows by a straightforward but tedious calculation based on \eqref{eq:CR} and $I_a$ being an involution (cf.~\cite{Garde_2017}), we only need to consider the proof of part~(i). In the same manner as above, it can be argued that $I_{a}$ maps $S(0,r)$, with $0<r<1$, onto some sphere $S(ce_a,R) \subset \brum$ and that $\overline{\brum_a}\cap S(ce_a,R)=\{I_{a}(-re_a), I_{a}(re_a)\}$, where
	\begin{align*}
	  I_{a}(re_a) &= (\rho^{-2}-1)\frac{r-\rho^{-1}}{(\rho^{-1}-r)^2}e_a + \rho^{-1} e_a = \frac{\rho-r}{1-\rho r}e_a, \\
	I_{a}(-re_a) &= -(\rho^{-2}-1)\frac{r+\rho^{-1}}{(\rho^{-1}+r)^2}e_a + \rho^{-1} e_a= \frac{\rho+r}{1+\rho r}e_a.
	\end{align*}  
	Hence, $C = ce_a = (I_{a}(-re_a) + I_{a}(re_a))/2$ and $R = \abs{I_{a}(-re_a)-C}$, giving the expressions in \eqref{eq:CR}.

    As mentioned above, the two relations between $a$, $r$, $C$, and $R$ in \eqref{eq:CR} are equivalent to those in \eqref{eq:ar}, and so we can also assume the knowledge of the latter in the following. Since $B(ce_a,R)\subseteq I_{a}(\brum) = \brum$, it must hold $0<c<1-R$. With the help of the latter equality in \eqref{eq:ar}, we thus get
	\begin{equation*}
	\abs{a-ce_a} = \abs{\rho - c} = \Big\lvert\frac{c}{1 - Rr} - c\Big\rvert = \frac{cr}{1-Rr}R < \frac{1-R}{1-Rr}R < R
	\end{equation*}
    as $0<r<1$. In other words, $I_a(0) = a \in B(ce_a,R)$. As $I_a$ is a continuous involution, it thus maps the whole of $B(0,r)$ onto $B(ce_a,R)$ and the proof is complete.   
\end{proof}

As expected, the Jacobian matrix of $I_a$ is denoted $J_a := J_{\as,b}$, with $J_{\as,b}$ explicitly given in Proposition~\ref{prop:generalJac}. The following corollary provides information about the behavior of $J_a$ on $\partial \brum$, enabling substitutions corresponding to $I_a$ in integrals over $\partial \brum$. In particular, it enables the introduction of the distributional Kelvin transformation on $\D'(\partial \brum)$, in the sense of distributions on a smooth manifold (cf.~e.g.~\cite[Chapter~6.3]{HormanderI}).
\begin{corollary} \label{coro:Jbdry}
	For $x\in\partial\brum$, $J_a(x)x = g_a^2(x)I_a(x)$ and $J_a(x)I_a(x) = g_a^2(x)x$.
\end{corollary}
\begin{proof}
  The first equality follows from a combination of Proposition~\ref{prop:generalJac} and Theorem~\ref{thm:balls} that relate both $I_a(x)$ and $J_a(x)$ to $\ident-2P_{x-\as}$ when $x \in \partial \brum$. By applying the inverse of $J_a(x)$ to this first equality, one obtains
  \begin{equation*}
  	x = g_a^2(x) J_a^{-1}(x) I_a(x) = g_a^{-2}(x) J_a(x) I_a(x), \qquad x \in \partial \brum,
  \end{equation*}
  where we used Proposition~\ref{prop:generalJac}(ii) in the second step. This completes the proof.
\end{proof}

The linear maps $K_a$, $\mathcal{I}_a$, and $G_a$ can obviously also be interpreted as operators on $L^1_{\rm loc}(\partial \brum) = L^1(\partial \brum)$, and $G_a$ as such on $\D'(\partial \brum)$ as well. In the spirit of \eqref{eq:kelvindist} and \eqref{eq:invdist}, we introduce the extensions $K_a, \mathcal{I}_a:  \D'(\partial \brum) \to \D'(\partial \brum)$ via
\begin{align*} 
 \inner{K_{a}f, \varphi} &:= \inner{f,G_{a}^2 K_{a}\varphi}, \\[1mm]
 \inner{\mathcal{I}_{a}f, \varphi} &:= \inner{f,G_{a}^{2d-2} \mathcal{I}_{a}\varphi}, 
\end{align*}
for any $f \in \D'(\partial \brum)$ and all $\varphi \in \Ccinf(\partial \brum)= C^\infty(\partial \brum)$. As in the case of \eqref{eq:kelvindist} and \eqref{eq:invdist}, these definitions make sense because $C^\infty(\partial \brum) =  G_{a}^{k} \mathcal{I}_{a}(C^\infty(\partial \brum))$ for any $k \in \mathbb{Z}$, and it is also easy to check that the extensions still satisfy \eqref{eq:IGcommut} and \eqref{eq:KGcommut}. 

\begin{remark}
	The extended operators coincide with the standard ones \eqref{eq:Kelvin} and \eqref{eq:Inversion} on $L^1(\partial\brum)$. Indeed, one can prove this by  performing  similar calculations as in \eqref{eq:kelvindist_mot}, with the exception that this time around the (boundary) Jacobian determinant reads 
	\begin{equation}
	\label{eq:boundary_Jaco}
	\frac{\abs{\det J_a(x)}}{\abs{J_a(x) x}} =  \frac{g_a^{2d}(x)}{\abs{g_a^2(x) I_a(x)}} = g_a^{2d-2}(x), \qquad x \in \partial \brum,
	\end{equation}
	due to Corollary~\ref{coro:Jbdry} and since $\nu(x) = x$ is the exterior unit normal at $x \in \partial \brum$. For $K_{a}$ such a calculation is actually explicitly carried out in the proof of Lemma~\ref{lemma:opnorms} below. As for the Sobolev spaces over the domain $\Omega$ in \eqref{eq:Kelvin_norm}, it follows straightforwardly  that
	\begin{equation} \label{eq:bKelvin_bound}
	\norm{G_{a}^{k} \mathcal{I}_{a} f}_{H^s(\partial \brum)} \leq C(a,k,s) \norm{f}_{H^s(\partial \brum)}
	\end{equation}
	for any $s, k \in \mathbb{Z}$. The standard theory on interpolation of Sobolev spaces demonstrates that \eqref{eq:bKelvin_bound} actually holds for any $s \in \R$; see,~e.g.,~\cite{sobolev,Lions1972}. Finally, it follows via a density argument that $K_a$ and $\mathcal{I}_a$ are involutions on $H^s(\partial \brum)$ for any $s \in \R$.
\end{remark}

We have now gathered enough tools to explicitly evaluate certain operator norms of $\Ka$ and closely related operators. 
\begin{lemma} \label{lemma:opnorms}
	The following results hold in $\mathscr{L}(L^2(\brum))$ and $\mathscr{L}(L^2(\partial\brum))$:
	\begin{enumerate}[(i)]
		\item $G_a^2 \Ka$ is an isometry in $\mathscr{L}(L^2(\brum))$ and $\Ga \Ka$ is an isometry in $\mathscr{L}(L^2(\partial\brum))$.
		\item $K_a^* = G_a^4\Ka$ in $\mathscr{L}(L^2(\brum))$ and $K_a^* = G_a^2 \Ka$ in $\mathscr{L}(L^2(\partial\brum))$.
		\item There are the following operator norm equalities:
		\begin{align*}
		\norm{\Ka}_{\mathscr{L}(L^2(\brum))} &= \norm{G_a^4\Ka}_{\mathscr{L}(L^2(\brum))} = \norm{\Ka}_{\mathscr{L}(L^2(\partial\brum))}^2 =  \norm{G_a^2\Ka}_{\mathscr{L}(L^2(\partial\brum))}^2 \\
		&= \norm{G_a^{\pm 1}}_{\mathscr{L}(L^2(\brum))}^2 = \norm{G_a^{\pm 1}}_{\mathscr{L}(L^2(\partial\brum))}^2 = \frac{1+\rho}{1-\rho}.
		\end{align*}
	\end{enumerate}
\end{lemma}

\begin{proof}
 Proof of part (i). It follows directly from \eqref{eq:Kelvin}, the change of variables formula, Proposition~\ref{prop:generalJac}(i), and \eqref{eq:IGcommut} that
 \begin{equation} \label{eq:isometry}
 	\int_{\brum} \abs{G_a^2 K_a f}^2\,\di x =  \int_{\brum} G_a^{2d} \mathcal{I}_a\abs{f}^2\,\di x = \int_{\brum} G_a^{2d} \mathcal{I}_a G_a^{2d} \mathcal{I}_a \abs{f}^2\,\di x = \int_{\brum} \abs{f}^2\,\di x
 \end{equation}
 for all $f\in L^2(\brum)$, which proves the first half of the claim.
 To prove the second half, observe that \eqref{eq:boundary_Jaco} yields
 \begin{equation} \label{eq:bndchva}
	\int_{\partial\brum} \ia g  \,\di S = \int_{\partial\brum} G_a^{2d-2} g \, \di S \qquad \text{and}  \qquad  \int_{\partial\brum} g  \,\di S = \int_{\partial\brum} G_a^{2d-2} \ia g \, \di S
 \end{equation}
 for all $g\in L^1(\partial\brum)$. As in \eqref{eq:isometry}, one thus obtains
 \begin{equation*}
   	\int_{\partial\brum} \abs{G_a K_a f}^2 \,\di S  =  \int_{\partial\brum} \abs{f}^2 \,\di S
 \end{equation*}
 for all $f\in L^2(\partial \brum)$, which completes the proof of part (i).
	
 Proof of part (ii). We only prove the result in $\mathscr{L}(L^2(\partial\brum))$ since the proof for $\mathscr{L}(L^2(\brum))$ follows from the same line of reasoning. For any $f,g\in L^2(\partial\brum)$,
 \begin{align*}
 	\int_{\partial\brum} \Ka f \, \overline{g}\,\di S &= \int_{\partial\brum} \ia f \, \overline{G_a^{d-2}g}\,\di S = \int_{\partial\brum} f\overline{G_a^{2d-2}\ia G_a^{d-2}g}\,\di S \\
	&= \int_{\partial\brum} f \, \overline{G_a^{d} \ia g}\,\di S = \int_{\partial\brum} f\overline{G_a^{2}\Ka g}\,\di S,
 \end{align*}
 where we employed \eqref{eq:bndchva} and \eqref{eq:IGcommut}.
	
 Proof of part (iii). Since $\as\not\in\brum$, it is straightforward to see that both the maximum and the minimum of $\ga$ over the compact set $\overline{\brum}$ are attained on $\partial\brum$.\footnote{For $d>2$, this follows from the maximum principle and the fact that $g_a^{d-2}$ is harmonic.} More precisely, the maximum (or the minimum) is obviously found at the point closest to (or furthest from) $\as$,~i.e.~at $e_a$ (or $-e_a$), which leads to
 \begin{align*}
	\sup_{x\in\brum}g_a^2(x)  = \frac{\rho^{-2}-1}{\abs{e_a-\as}^2} = \frac{(\rho^{-1}-1)(\rho^{-1}+1)}{(\rho^{-1}-1)^2} = \frac{1+\rho}{1-\rho}, \\[1mm]
	\inf_{x\in\brum}g_a^2(x) = \frac{\rho^{-2}-1}{\abs{e_a+\as}^2} = \frac{(\rho^{-1}-1)(\rho^{-1}+1)}{(\rho^{-1}+1)^2} = \frac{1-\rho}{1+\rho}.
 \end{align*}
 In particular, as the norm of a multiplication operator on $L^2$ is given by the essential supremum of the multiplier, we have
 \begin{equation*}
 	\norm{G_a^{\pm 1}}_{\mathscr{L}(L^2(\brum))} = \norm{G_a^{\pm 1}}_{\mathscr{L}(L^2(\partial\brum))} = \left( \frac{1+\rho}{1-\rho} \right)^{1/2}.
 \end{equation*}
 Due to the duality results of part~(ii), what remains to be shown is $\norm{\Ka}_{\mathscr{L}(L^2(\brum))} = \norm{G_a^2}_{\mathscr{L}(L^2(\brum))}$ and $\norm{\Ka}_{\mathscr{L}(L^2(\partial\brum))} = \norm{\Ga}_{\mathscr{L}(L^2(\partial\brum))}$. Again, the proofs of these identities are analogous, and we only show the latter. By virtue of part~(i) and \eqref{eq:KGcommut},
 \begin{equation*}
	\norm{\Ka f}_{L^2(\partial\brum)} = \norm{\Ga \Ka\Ga f}_{L^2(\partial\brum)} = \norm{\Ga f}_{L^2(\partial\brum)} 
 \end{equation*}
 for all $f\in L^2(\partial\brum)$, which concludes the proof.
\end{proof}

In addition to Theorem~\ref{thm:kelvinlaplace}, we must also consider the commutation of $\Ka$ and $\nabla$, as this is needed for handling Neumann boundary values. Unfortunately, the resulting expression is somewhat more complicated than that in~Theorem~\ref{thm:kelvinlaplace}. First of all, note that 
\begin{align*}
	\nabla \big(g_a^m(x) \big) &= m g_a^{m-1}(x)\nabla \ga(x) = -m g_a^m(x)\frac{x-\as}{\abs{x-\as}^2}, \qquad  m\in\rum{N}_0.
\end{align*}
Hence, for any $u \in C^\infty(\overline{\brum})$,
\begin{align}
\nabla \Ka u(x) &= g_a^{d-2}(x)\nabla\ia u(x) + \ia u(x)\nabla g_a^{d-2}(x) \notag \\[1mm]
&= J_a^{\rm T}(x)\Ka\nabla u(x) + (2-d)\frac{x-\as}{\abs{x-\as}^2}\Ka u(x), \label{eq:nablaKa}
\end{align}
where $\Ka$ is separately applied to each component of $\nabla u$. As $\as \notin \overline{\brum}$, the expression \eqref{eq:nablaKa} is valid for all $x \in \overline{\brum}$; recall that $C^\infty(\overline{\brum})$ consists of the restrictions of the elements in $C^\infty(\R^d)$ to $\overline{\brum}$.

Observe that $\nu(x) = x$ is the exterior unit normal to $\partial \brum$ for any $x\in\partial\brum$. According to \eqref{eq:nablaKa},
\begin{align}
\nu \cdot \nabla \Ka u(x) &= G_a^d\ia (\nu\cdot\nabla u)(x) + (2-d)x\cdot\frac{x-\as}{\abs{x-\as}^2}\Ka u(x) \notag\\[1mm]
&= G_a^2\Ka(\nu\cdot\nabla u)(x)+ (2-d)H_a\Ka u(x), \label{eq:kelvinnormalderiv}
\end{align}
for all $x \in \partial \brum$ and $u \in C^\infty(\overline{\brum})$. Here, the multiplication operator $H_a \in \mathscr{L}(H^s(\partial \brum))$, $s \in \R$, is defined by
\begin{equation*}
H_a f(x) := x\cdot\frac{x-\as}{\abs{x-\as}^2} f(x),
\end{equation*}
and the first equality in \eqref{eq:kelvinnormalderiv} follows from the identity
\begin{equation*}
x \cdot J_a^{\rm T}(x) K_a \nabla u(x) = g_a^{d}(x) I_a(x) \cdot \mathcal{I}_a \nabla u(x) = g_a^d(x) \mathcal{I}_a (\nu \cdot \nabla u)(x), \qquad x \in \partial \brum,
\end{equation*}
that is based on Corollary~\ref{coro:Jbdry}.

Let $U = W\setminus\partial\brum$ where $W\subseteq\overline{\brum}$ is a relatively open neighborhood of $\partial \brum$. The identity \eqref{eq:kelvinnormalderiv} extends by continuity to all $u$ in 
\begin{equation*}  
H^1_{\Delta}(U) := \{ v \in H^1(U) \mid \Delta v \in L^2(U) \}
\end{equation*}
equipped with the graph norm, that is,
\begin{equation}
  \label{eq:kelvinnormalderiv2}
( \nu \cdot \nabla \Ka u)|_{\partial \brum} = G_a^2\Ka\big((\nu\cdot\nabla u)|_{\partial \brum}\big)+ (2-d)H_a\Ka (u|_{\partial \brum}), \qquad u \in H^1_{\Delta}(U).
\end{equation}
This result follows from $C^\infty(\overline{U})$ being dense in $H^1_{\Delta}(U)$~\cite{Lions1972}, the Neumann trace extending to a bounded operator $H^1_{\Delta}(U) \to H^{-1/2}(\partial \brum)$~\cite[Lemma~1,~p.~381]{Dautray1988}, the standard trace theorem, and the boundedness of $K_a$ on  $H^{\pm 1/2}(\partial \brum)$ and $H^1_\Delta(U)$ guaranteed by \eqref{eq:bKelvin_bound}, \eqref{eq:Kelvin_norm}, and Theorem~\ref{thm:kelvinlaplace}.

The most essential message of \eqref{eq:kelvinnormalderiv2} is that a Neumann condition on $\partial\brum$ is transformed by $\Ka$ into a Robin condition. Luckily, the Robin condition transforms back to a Neumann condition for difference measurements of EIT, as revealed in the next section.

\section{Application to electrical impedance tomography} \label{sec:EIT}

Let $\Omega_{C,R} := \brum \setminus \overline{B(C,R)}$ for $C\in\brum$ and $R\in (0,1-\abs{C})$. We only consider the case of a perfectly conducting inclusion $B(C,R)$ (formally with conductivity $\infty$) and assume that  $\Omega_{C,R}$ is characterized by unit conductivity. Hence, if the electric potential on the exterior boundary is set to $f\in H^{1/2}(\partial\brum)$, then the interior potential $u \in H^1(\Omega_{C,R})$ is the unique solution to
\begin{equation}
  \label{eq:forward}
\begin{split}
\Delta u &= 0 \quad\text{in } \Omega_{C,R}, \\
u &= \begin{cases} 0 &\text{ on } \partial B(C,R), \\ f &\text{ on } \partial\brum. \end{cases}
\end{split}
\end{equation}
We define the DN map associated to the inclusion $B(C,R)$ as
\begin{equation*}
\Lambda_{C,R} : f\mapsto \nu \cdot \nabla u|_{\partial\brum}, \qquad H^{1/2}(\partial\brum)\to H^{-1/2}(\partial\brum),
\end{equation*}
and note that it is well known to be bounded due to the continuous dependence of the solution to \eqref{eq:forward} on the Dirichlet data and a suitable Neumann trace theorem (cf.,~e.g.,~\cite[Lemma~1,~p.~381]{Dautray1988}). We also define $\Lambda_1: H^{1/2}(\partial\brum)\to H^{-1/2}(\partial\brum)$ to be the DN map for the inclusion-free problem
\begin{align*}
\Delta w &= 0 \quad\text{in } \brum, \\
w &=  f \quad\text{on } \partial\brum.
\end{align*}

For each $C\in\brum\setminus\{0\}$ and $R\in (0,1-\abs{C})$, we choose the unique $a = a(C,R) \in\brum\setminus\{0\}$ such that $I_a(B(C,R)) = B(0,r)$ and $I_a(B(0,r)) = B(C,R)$ for some $R < r < 1$, which is possible by virtue of Theorem~\ref{thm:balls}. We consistently use this connection between $C,R$ and $a,r$ in what follows. The accordingly Kelvin-transformed potential $\tilde{u} := \Ka u \in H^1(\Omega_{0,r})$ is the unique solution of
\begin{equation}
  \label{eq:forward2}
\begin{split}
\Delta \tilde{u} &= 0 \quad\text{in } \Omega_{0,r}, \\
\tilde{u} &= \begin{cases} 0 &\text{ on } \partial B(0,r), \\ \tilde{f} &\text{ on } \partial\brum, \end{cases}
\end{split}
\end{equation}
for $\tilde{f} := K_a f$. Indeed, the first equation is a direct consequence of Theorem~\ref{thm:kelvinlaplace}. To prove the second one, observe that obviously
\begin{equation*}
	(K_a v)|_{\partial \Omega_{0,r}} = K_a(v|_{\partial \Omega_{C,R}})
\end{equation*}
for all $v \in C^{\infty}(\overline{\Omega_{C,R}})$, and this equality extends by density for any $v \in H^1(\Omega_{C,R})$ due to (obvious generalizations of) the estimates \eqref{eq:Kelvin_norm}, \eqref{eq:bKelvin_bound}, and the continuity of the Dirichlet trace map on both $H^1(\Omega_{C,R})$ and $H^1(\Omega_{0,r})$. Since $K_a: H^1(\Omega_{0,r}) \to H^1(\Omega_{C,R})$ is obviously the inverse of $K_a: H^1(\Omega_{C,R}) \to H^1(\Omega_{0,r})$, the solution of \eqref{eq:forward} can alternatively be written as $K_a \tilde{u}$ with $\tilde{u}$ being the solution to \eqref{eq:forward2}.

\begin{remark}
	The physically correct condition at a perfectly conducting inclusion, is that the potential equals such a constant on its boundary that the corresponding normal current density has zero mean, under the sound assumption that there are no sinks or sources inside the inclusion. Notice that this constant may depend on both $f$ and the inclusion itself. Let $\tilde{\Lambda}_{C,R}$ be the DN map corresponding to the boundary conditions of such a physically accurate setting. Since we only have a single connected inclusion,\footnote{In general, different constants appear on each connected component of a perfectly conducting inhomogeneity.} $\tilde{\Lambda}_{C,R} = \Lambda_{C,R}|_{\mathcal{Y}_{C,R}}$, where $\mathcal{Y}_{C,R}$ is a linear subspace of $H^{1/2}(\partial\brum)$, which again may depend on the inclusion. Due to this inconvenience --- in particular, for the inverse conductivity problem where the inclusion is not known a priori --- the DN operator with a larger domain $\Lambda_{C,R}$ is often investigated instead of $\tilde{\Lambda}_{C,R}$ (cf.,~e.g.,~\cite{Ammari2007a,Borman_2009,Erhard2003,Kress_2012,Kress_2011,Seo_2002}). This is also the choice in this work. 
\end{remark}
We are now ready to prove an explicit relation between the DN maps for the concentric and nonconcentric geometries.
\begin{theorem} \label{thm:DNkelvin}
 Let $B(C,R) = I_a(B(0,r))$ for $a\in\brum\setminus\{0\}$ and $r\in(0,1)$. Then,
	\begin{equation*}
	\Lambda_{C,R} = G_a^2\Ka\Lambda_{0,r}\Ka + (2-d)H_a.
	\end{equation*}
	Furthermore,
	\begin{align*}
	  \Lambda_{C,R}-\Lambda_1 &= G_a^2\Ka(\Lambda_{0,r}-\Lambda_1)\Ka,
        \end{align*}
        or equivalently,
        \begin{align*}
	\Lambda_{0,r}-\Lambda_1 &= G_a^2\Ka(\Lambda_{C,R}-\Lambda_1)\Ka.
	\end{align*}
\end{theorem}
\begin{proof}
	Let $\tilde{g} := (\nu\cdot\nabla \tilde{u})|_{\partial\brum} \in H^{-1/2}(\partial \brum)$ be the normal current density for the solution of \eqref{eq:forward2}. We obtain directly from \eqref{eq:kelvinnormalderiv2} that
	\begin{align*}
	\Lambda_{C,R} f &= (\nu\cdot \nabla u)|_{\partial\brum} = (\nu\cdot \nabla \Ka \tilde{u})|_{\partial\brum} = G_a^2\Ka \tilde{g} + (2-d) H_a\Ka \tilde{f} \nonumber \\[1mm]
	&= G_a^2\Ka \Lambda_{0,r} \tilde{f} + (2-d) H_a\Ka \tilde{f} = \big(G_a^2\Ka\Lambda_{0,r}\Ka + (2-d)H_a \big)f
	\end{align*}
	for all $f\in H^{1/2}(\partial\brum)$. This proves the first part of the claim.

    Following exactly the same line of reasoning as above,
	\begin{equation*}
	\Lambda_1 = G_a^2\Ka\Lambda_1 \Ka + (2-d)H_a,
	\end{equation*}
        which shows the claimed representation for $\Lambda_{C,R}-\Lambda_1$. As $G_a^2 \Ka$ and $\Ka$ are involutions on $C^\infty(\partial \brum)$ (cf.~\eqref{eq:KGcommut}), they are also such on any $H^s(\partial \brum)$, $s \in \R$, by density and \eqref{eq:bKelvin_bound}. Hence, the representation for $\Lambda_{0,r}-\Lambda_1$ follows directly from that of $\Lambda_{C,R}-\Lambda_1$, and the proof is complete.
\end{proof}

Before determining the spectrum of $\Lambda_{0,r}-\Lambda_1$ needed in proving our main result, we briefly review a few facts about spherical harmonics; see,~e.g.,~\cite{Axler_2001,Efthimiou_2014,Gurarie} for additional details. Denoting $\rum{R}^d\setminus\{0\}\ni x = \eta\theta$ with $\eta = \abs{x}$ and $\theta = \frac{x}{\abs{x}}\in\partial\brum$, it is well known that the Laplace operator can be written in polar coordinates (cf.\ \cite[Section 4.5]{Gurarie}) as
\begin{equation}
\Delta = \frac{1}{\eta^{d-1}}\partial_\eta(\eta^{d-1}\partial_\eta) + \frac{1}{\eta^2}\Delta_{\partial\brum}, \label{eq:laplacepolar}
\end{equation}
where $\Delta_{\partial\brum}$ is the Laplace--Beltrami operator on $\partial\brum$ with respect to $\theta$.

A polynomial $p$ on $\rum{R}^d$ is called \emph{homogeneous of degree $n$} if $p(x) = \smash{\sum_{\abs{\alpha}=n}} c_\alpha x^\alpha$ with scalar coefficients $c_\alpha$, or equivalently $p(tx) = \smash{t^np(x)}$ for $t\in\rum{R}$; following the standard notation, $\alpha \in \smash{\mathbb{N}_0^d}$ is here a multi-index, $\abs{\alpha} := \smash{\sum_{j = 1}^{d}} \alpha_j$, and $x^\alpha := \smash{\Pi_{j=1}^d x_j^{\alpha_j}}$. The complex vector space $\mathcal{H}_{n,d}$, of \emph{spherical harmonics of degree $n$}, comprise the \emph{harmonic} polynomials homogeneous of degree $n$ restricted to $\partial\brum$,~i.e.,
\begin{equation*}
\mathcal{H}_{n,d} := \biggl\{ p|_{\partial\brum} \mid p(x) = \sum_{\abs{\alpha}=n}c_\alpha x^\alpha,\enskip x\in\rum{R}^d, \enskip \Delta p = 0 \biggr\}.
\end{equation*}
The corresponding dimension $\alpha_{n,d} := \dim(\mathcal{H}_{n,d})$ is given by
\begin{equation*}
\alpha_{n,d} = \
\binom{n+d-1}{d-1} - \binom{n+d-3}{d-1}, \\
\end{equation*}
where we use the convention $\binom{m}{k} = 0$ for $m<k$. 

The eigenvalues of $\Delta_{\partial\brum}$ are $\tilde{\lambda}_n := -n(n+d-2)$, $n\in\rum{N}_0$, with the algebraic and geometric multiplicity~$\alpha_{n,d}$. The eigenspace corresponding to the eigenvalue $\tilde{\lambda}_n$ is $\mathcal{H}_{n,d}$, spanned by the orthonormal $n$th degree spherical harmonics $\{f_{n,j}\}_{j=1}^{\alpha_{n,d}}$. The set of $f_{n,j}$ for all $n\in\rum{N}_0$ and $j\in\{1,\dots,\alpha_{n,d}\}$ is an orthonormal basis for $L^2(\partial\brum)$,~i.e.~$L^2(\partial\brum) = \bigoplus_{n=0}^\infty \mathcal{H}_{n,d}$.

Using separation of variables and classic Sturm--Liouville theory for \eqref{eq:laplacepolar}, it is known that any harmonic function $u$ on $\brum$ or $\Omega_{0,r}$ can be written as 
\begin{equation*}
u(x) = \sum_{n=0}^\infty\sum_{j = 1}^{\alpha_{n,d}} c_{n,j} R_{n}(\eta)f_{n,j}(\theta)
\end{equation*}
for $(c_{n,j})\in\ell^2$. Here $R_{n}$ is a solution of 
\begin{equation}
\eta^2 R_n''(\eta) + (d-1)\eta R_n'(\eta) + \tilde{\lambda}_n R_n(\eta) = 0, \label{eq:eulerode}
\end{equation}
with suitable boundary conditions at
$\eta = 1$ and either at $\eta=0$ or $\eta = r$ depending on those required from $u$. If $u|_{\partial\brum} = f_{n,j}$, then $R_n(1)=1$ and $(\nu\cdot\nabla u)|_{\partial\brum} = R_n'(1)f_{n,j}$, i.e.,~$R_n'(1)$ is the $n$th eigenvalue of the associated DN map corresponding to the eigenfunction~$f_{n,j}$. In particular, the algebraic and geometric multiplicity of $R_n'(1)$ is also $\alpha_{n,d}$. Based on these observations, we can explicitly determine the eigenvalues of $\Lambda_{0,r}$, $\Lambda_1$, and $\Lambda_{0,r} - \Lambda_1$.

\begin{proposition} \label{prop:eigvals}
	The eigenvalues of $\Lambda_1$ are $\{n\}_{n\in\rum{N}_0}$ and those of $\Lambda_{0,r}$ are, for $n\in\rum{N}_0$,
	\begin{equation*}
	\hat{\lambda}_n := \begin{dcases}
	\frac{n+(n+d-2)r^{2n+d-2}}{1-r^{2n+d-2}}, & \quad d>2 \vee n\geq 1,\\
	-\frac{1}{\log(r)}, & \quad d = 2 \wedge n=0,
	\end{dcases}
	\end{equation*} 
	both with the algebraic and geometric multiplicity $\alpha_{n,d}$ and the eigenspace $\mathcal{H}_{n,d}$. 
	As a consequence, the eigenvalues of $\Lambda_{0,r}-\Lambda_1$ are
	\begin{equation*}
	\lambda_n := \begin{dcases}
	\frac{2n+d-2}{r^{2-d-2n}-1}, & \quad d>2 \vee n\geq 1,\\
	-\frac{1}{\log(r)}, & \quad d = 2 \wedge n=0,
	\end{dcases}
	\end{equation*}
    also with the algebraic and geometric multiplicity $\alpha_{n,d}$ and the eigenspace $\mathcal{H}_{n,d}$.  
\end{proposition}
\begin{proof}
	Note that \eqref{eq:eulerode} is a second order Cauchy--Euler equation. Its indicial equation 
	\begin{equation*}
	m_n^2 + (d-2)m_n + \tilde{\lambda}_n = 0
	\end{equation*}
	has the solutions
	\begin{equation*}
	m_n = \begin{cases}
	n  & =: m_n^+, \\
	2-d-n  & =: m_n^-.
	\end{cases}
	\end{equation*}
	If $d > 2$, then $m_n^+ \neq m_n^-$ for all $n\in\rum{N}_{0}$, which also holds for $d=2$ when $n\geq 1$. Hence, all solutions to \eqref{eq:eulerode} are of the form
	\begin{equation*}
	R_n(\eta) = \begin{cases} c_n \eta^{n} + \tilde{c}_n \eta^{2-d-n}, & \quad d>2 \vee n\geq 1, \\[1mm]
	c_0 + \tilde{c}_0\log(\eta), & \quad d = 2 \wedge n=0.
	\end{cases}
	\end{equation*}
	Starting with $\Lambda_1$, we see that the boundary conditions for \eqref{eq:eulerode} are $\limsup_{\eta\to 0^+}\abs{R_n(\eta)} < \infty$ and $R_n(1) = 1$, which immediately imply $\tilde{c}_n = 0$ and $c_n = 1$ for all $n\in\rum{N}_0$. The eigenvalues of $\Lambda_1$ are thus $R_n'(1) = n$ for $n\in\rum{N}_0$, as claimed. 
	
	Now considering $\Lambda_{0,r}$, the boundary conditions for \eqref{eq:eulerode} become $R_n(r) = 0$ and $R_n(1) = 1$. For the special case $d = 2$ and $n=0$, we arrive at $R_0(\eta) = 1-\smash{\frac{\log(\eta)}{\log(r)}}$. For $d>2$ or $n\geq 1$, we obtain
	\begin{equation*}
	R_n(\eta) = \frac{\eta^n}{1-r^{2n+d-2}} + \frac{\eta^{2-d-n}}{1-r^{2-d-2n}}.
	\end{equation*}
	Evaluating $R_n'(1)$ provides the sought-for representation for $\hat{\lambda}_n$. Furthermore, as the eigenfunctions of $\Lambda_{0,r}$ and $\Lambda_1$ coincide, we obtain the representation for $\lambda_n$ by evaluating the difference $R_n'(1)-n$ for $n\in\rum{N}_0$.
\end{proof}

\begin{remark} \label{remark:eigval}
	The eigenvalues $\lambda_n$, $n \in \mathbb{N}_0$, of $\Lambda_{0,r}- \Lambda_1$ given in Proposition~\ref{prop:eigvals} decay strictly in $n$. Indeed, the derivative of the function $y\mapsto \frac{y}{r^{-y}-1}$ reveals that  the claim holds for $n \in \mathbb{N}_0$ if $d \geq 3$ and for $n \in \mathbb{N}$ if $d=2$, because
	\begin{equation*}
	1+\log(r^{2n+d-2}) < r^{2n+d-2} \qquad \text{when} \  2n+d-2 > 0,
	\end{equation*}
	since $r\in (0,1)$. For $d = 2$, we see that $\lambda_0 > \lambda_1$ if and only if $1+\log(r^{-2}) < r^{-2}$, which holds as $r^{-2}>1$. Note that the observed strict decay is in contrast to the case of an inclusion with finite conductivity, where the eigenvalues may exhibit an initial increase in magnitude before decaying with respect to the ordering of the spherical harmonics \cite[Remark~3.2]{Garde_2017}.
\end{remark}

For completeness, the following lemma shows that the limit behavior $\lim_{n\to\infty}\lambda_n = 0$ guarantees that the difference map $\Lambda_{C,R}-\Lambda_1$ extends to a compact self-adjoint operator on $L^2(\partial\brum)$. However, it is actually well known that $\Lambda_{C,R}-\Lambda_1$ is smoothening because $\Lambda_{C,R}$ and $\Lambda_1$ are pseudodifferential operators (modulo a smoothing operator) with the same symbol~\cite{Lee1989}.

\begin{lemma} \label{lemma:DNextend}
	For each ball satisfying $\overline{B(C,R)}\subset \brum$, the operator $\Lambda_{C,R}-\Lambda_1$ continuously extends to a compact self-adjoint operator in $\mathscr{L}(L^2(\partial\brum))$.
\end{lemma}
\begin{proof}
	The proof is completely analogous to \cite[Lemma 3.3]{Garde_2017} that should be consulted for further details. 
	
	For $r\in(0,1)$, the eigenvalues $\{\lambda_n\}_{n\in\rum{N}_0}$ of $\Lambda_{0,r}-\Lambda_1$ are bounded and satisfy $\lim_{n\to\infty}\lambda_n = 0$. Combined with the corresponding eigenfunctions $\{f_{n,j}\}$ forming an orthonormal basis for $L^2(\partial\brum)$ and $\Lambda_{0,r}-\Lambda_1$ being symmetric in the $L^2(\partial\brum)$-inner product, it follows that $\Lambda_{0,r}-\Lambda_1$ continuously extends to a compact self-adjoint operator in $\mathscr{L}(L^2(\partial\brum))$. 
	
	For any $C\in\brum\setminus\{0\}$ and $R \in (0,1-\abs{C})$, we may choose $a\in\brum\setminus\{0\}$ such that $I_a(B(C,R)) = B(0,r)$ for some $r\in(0,1)$. Hence, Theorem~\ref{thm:DNkelvin} and Lemma~\ref{lemma:opnorms} imply that $\Lambda_{C,R}-\Lambda_1$ inherits compactness and self-adjointness from $\Lambda_{0,r}-\Lambda_1$.
\end{proof}

\section{Depth-dependent norm bounds} \label{sec:bounds}

In this section we can finally present some depth-dependent norm estimates. However, it is convenient to first introduce certain weighted $L^2$-spaces on $\partial \brum$.
\begin{definition} \label{def:weightednorms}
	For $a\in\brum\setminus\{0\}$ and $s\in\rum{R}$, we denote by $L^2_{a,s}(\partial\brum)$ the weighted $L^2(\partial\brum)$-space equipped with the inner product and norm 
	\begin{equation*}
	\inner{f,g}_{a,s} := \inner{G_a^s f,G_a^s g}_{L^2(\partial\brum)} \qquad {\rm and} \qquad \norm{f}_{a,s} := \sqrt{\inner{f,f}_{a,s}}, \qquad f,g \in L^2(\partial\brum),
	\end{equation*}
respectively. Furthermore, we denote  by $\norm{\,\cdot\,}_{a,s,t}$ the operator norm of $\mathscr{L}(L^2_{a,s}(\partial\brum),L^2_{a,t}(\partial\brum))$ for $s,t\in\rum{R}$.
\end{definition}
As $\as \notin \brum$, the function $g_a^s$ is bounded away from zero and infinity on $\partial\brum$. Hence, it is obvious that the topologies of $L^2(\partial\brum)$ and $L^2_{a,s}(\partial\brum)$ are the same for any $s \in \R$.
However, using these newly defined weighted norms, one obtains useful relations between the norms of concentric and nonconcentric DN maps.
\begin{theorem} \label{thm:norminvariant} Let $B(C,R) = I_a(B(0,r))$ for $a\in\brum\setminus\{0\}$ and $r\in(0,1)$. For any $s,t\in\rum{R}$,
	\begin{align*}
	  \norm{\Lambda_{C,R}-\Lambda_1}_{a,s,t} &= \norm{\Lambda_{0,r}-\Lambda_1}_{a,1-s,-1-t},
        \end{align*}
        or equivalently,
        \begin{align*} 
	\norm{\Lambda_{0,r}-\Lambda_1}_{a,s,t} &= \norm{\Lambda_{C,R}-\Lambda_1}_{a,1-s,-1-t}.
	\end{align*}
\end{theorem}
\begin{proof}
	The result follows from a direct calculation utilizing Theorem~\ref{thm:DNkelvin}, \eqref{eq:KGcommut}, Lemma~\ref{lemma:opnorms}(i), and $\Ka$ being an involution on $L^2(\partial\brum)$. To be more precise,
	\begin{align*}
	\norm{\Lambda_{C,R}-\Lambda_1}_{a,s,t} &= \norm{G_a^2\Ka(\Lambda_{0,r}-\Lambda_1)\Ka}_{a,s,t} \\
	&= \sup_{f\in L^2(\partial\brum)\setminus\{0\}} \frac{\norm{G_a^tG_a^2\Ka(\Lambda_{0,r}-\Lambda_1)\Ka f}_{L^2(\partial\brum)}}{\norm{G_a^s f}_{L^2(\partial\brum)}} \\
 	&= \sup_{f\in L^2(\partial\brum)\setminus\{0\}} \frac{\norm{G_a\Ka G_a^{-1-t}(\Lambda_{0,r}-\Lambda_1)f}_{L^2(\partial\brum)}}{\norm{G_a^s \Ka f}_{L^2(\partial\brum)}} \\       
	&= \sup_{f\in L^2(\partial\brum)\setminus\{0\}} \frac{\norm{G_a^{-1-t}(\Lambda_{0,r}-\Lambda_1)f}_{L^2(\partial\brum)}}{\norm{G_a \Ka G_a^{1-s}f}_{L^2(\partial\brum)}} \\
	&= \sup_{f\in L^2(\partial\brum)\setminus\{0\}} \frac{\norm{G_a^{-1-t}(\Lambda_{0,r}-\Lambda_1)f}_{L^2(\partial\brum)}}{\norm{G_a^{1-s}f}_{L^2(\partial\brum)}} \\[1mm]
	&= \norm{\Lambda_{0,r}-\Lambda_1}_{a,1-s,-1-t}.		
	\end{align*}
This also proves the second part of the claim by replacing $s$ and $t$ with $1-s$ and $-1-t$.
\end{proof}

\begin{remark}
Some natural choices in Theorem~\ref{thm:norminvariant} are $s = 1/2$ and $t = -1/2$ or $s = 1$ and $t = -1$, which result in
\begin{align*}
\norm{\Lambda_{C,R}-\Lambda_1}_{a,1/2,-1/2} &= \norm{\Lambda_{0,r}-\Lambda_1}_{a,1/2,-1/2}, \\[1mm]
\norm{\Lambda_{C,R}-\Lambda_1}_{a,1,-1} &= \norm{\Lambda_{0,r}-\Lambda_1}_{\mathscr{L}(L^2(\partial\brum))}, \\[1mm]
\norm{\Lambda_{0,r}-\Lambda_1}_{a,1,-1} &= \norm{\Lambda_{C,R}-\Lambda_1}_{\mathscr{L}(L^2(\partial\brum))}.
\end{align*}
In addition, the choice $s = 1$ and $t=-1$ leads to a natural diagonalization of $\Lambda_{C,R}-\Lambda_1$; see Appendix~\ref{sec:appB} for the precise formulation of this result.
\end{remark}

Finally, it is time to present our optimal depth-dependent distinguishability bounds.
\begin{theorem} \label{thm:L2bnds}
  Assume $B(C,R) = I_a(B(0,r))$ for $a\in\brum\setminus\{0\}$ and $r\in(0,1)$.	 Let $\{\lambda_{n}\}_{n\in\rum{N}_0}$ denote the set of eigenvalues for $\Lambda_{0,r}-\Lambda_1$, cf.~Proposition~\ref{prop:eigvals}. Then,
  \begin{equation}
    \label{eq:depth_bounds}
	\frac{1-\rho}{1+\rho} \leq \frac{\norm{\Lambda_{0,r}-\Lambda_1}_{\mathscr{L}(L^2(\partial\brum))}}{\norm{\Lambda_{C,R}-\Lambda_1}_{\mathscr{L}(L^2(\partial\brum))}} \leq   \left[\frac{(1-\rho^2)^2 d}{(1+\rho^2)^2 d +4\rho^2\frac{\lambda_1}{\lambda_0}\left(\frac{\lambda_1}{\lambda_0}+2\right)}\right]^{1/2} \leq \frac{1-\rho^2}{1+\rho^2}.
	\end{equation}
	Furthermore, these bounds are optimal in the sense that
	\begin{equation}
          \label{eq:depth_bounds_opti}
		\inf_{r\in(0,1)}\frac{\norm{\Lambda_{0,r}-\Lambda_1}_{\mathscr{L}(L^2(\partial\brum))}}{\norm{\Lambda_{C,R}-\Lambda_1}_{\mathscr{L}(L^2(\partial\brum))}} = \frac{1-\rho}{1+\rho}, \qquad \sup_{r\in(0,1)}\frac{\norm{\Lambda_{0,r}-\Lambda_1}_{\mathscr{L}(L^2(\partial\brum))}}{\norm{\Lambda_{C,R}-\Lambda_1}_{\mathscr{L}(L^2(\partial\brum))}} = \frac{1-\rho^2}{1+\rho^2}.
	\end{equation}
\end{theorem}
\begin{proof}
	Since $\ga$ is bounded from below and above by positive constants on $\partial\brum$, it follows that $G_a^{-1}(L^2(\partial\brum)\setminus\{0\}) = L^2(\partial\brum)\setminus\{0\}$. Hence, applying Theorem~\ref{thm:norminvariant} with $s=1$ and $t=-1$ yields
	\begin{align}
	\norm{\Lambda_{C,R}-\Lambda_1}_{\mathscr{L}(L^2(\partial\brum))} &= \norm{\Lambda_{0,r}-\Lambda_1}_{a,1,-1} \notag\\
	&= \sup_{f\in L^2(\partial\brum)\setminus\{0\}}\frac{\norm{G_a^{-1}(\Lambda_{0,r}-\Lambda_1)f}_{L^2(\partial\brum)}}{\norm{G_a f}_{L^2(\partial\brum)}} \notag\\
	&= \sup_{f\in L^2(\partial\brum)\setminus\{0\}}\frac{\norm{G_a^{-1}(\Lambda_{0,r}-\Lambda_1)G_a^{-1}f}_{L^2(\partial\brum)}}{\norm{f}_{L^2(\partial\brum)}} \notag\\
	&= \norm{G_a^{-1}(\Lambda_{0,r}-\Lambda_1)G_a^{-1}}_{\mathscr{L}(L^2(\partial\brum))}. \label{eq:normequality}
	\end{align}
	In consequence, we immediately obtain the lower bound
	\begin{equation*}
	\frac{\norm{\Lambda_{0,r}-\Lambda_1}_{\mathscr{L}(L^2(\partial\brum))}}{\norm{\Lambda_{C,R}-\Lambda_1}_{\mathscr{L}(L^2(\partial\brum))}}	\geq \norm{\Ga^{-1}}_{\mathscr{L}(L^2(\partial\brum))}^{-2} = \frac{1-\rho}{1+\rho},
	\end{equation*}
    where the equality follows from Lemma~\ref{lemma:opnorms}(iii). We postpone proving the optimality of this estimate till the end of this proof. 

    Let us then consider the upper bound in \eqref{eq:depth_bounds}. If $x\in \partial\brum$, then 
	\begin{equation} \label{eq:ga_m2}
		g_a^{-2}(x) = c_0 + c_1 f_1(x),
	\end{equation}
	with $c_0 := b^{-2}(1+\rho^{-2}) = (1+\rho^2)/(1-\rho^2)$, $c_1 := -2b^{-2}=-2\rho^2/(1-\rho^2)$, and $f_1(x) := x\cdot \as$. In particular, $g_a^{-2}$ is a weighted sum of spherical harmonics of degree zero and one. Hence, we deduce
	\begin{align}
	\norm{G_a^{-1}(\Lambda_{0,r}-\Lambda_1)G_a^{-1}}_{\mathscr{L}(L^2(\partial\brum))} &= \sup_{f\in L^2(\partial\brum)\setminus\{0\}} \frac{\norm{G_a^{-1}(\Lambda_{0,r}-\Lambda_1)G_a^{-1} f}_{L^2(\partial\brum)}}{\norm{f}_{L^2(\partial\brum)}} \notag\\
	&\geq \frac{\norm{g_a^{-1}(\Lambda_{0,r}-\Lambda_1)g_a^{-2} }_{L^2(\partial\brum)}}{\norm{g_a^{-1}}_{L^2(\partial\brum)}} \notag\\
	&= \frac{\norm{(c_0+c_1 f_1)^{1/2}(c_0\lambda_0 + c_1\lambda_1 f_1) }_{L^2(\partial\brum)}}{\norm{g_a^{-1}}_{L^2(\partial\brum)}}.  \label{eq:bnd1}
	\end{align}
	By symmetry, or using  integration formulas for polynomials on the unit sphere \cite{Folland_2001}, we have
	\begin{equation*}
	\norm{g_a^{-1}}_{L^2(\partial\brum)}^2 = \int_{\partial\brum} (c_0+c_1 x\cdot \as) \, \di S(x) = c_0 \abs{\partial\brum}. 
	\end{equation*}
    To also simplify the numerator on the right-hand side of \eqref{eq:bnd1}, we write
	\begin{align}
	\norm{(c_0+c_1 f_1)^{1/2}(c_0\lambda_0 + c_1\lambda_1 f_1) }_{L^2(\partial\brum)}^2 &= \int_{\partial\brum} (c_0+c_1 x\cdot\as)(c_0\lambda_0+c_1\lambda_1 x\cdot\as)^2 \, \di S(x) \notag \\
	&= c_0^3\lambda_0^2\abs{\partial\brum} + c_0c_1^2\lambda_1(\lambda_1+2\lambda_0) \int_{\partial\brum} (x\cdot\as)^2 \,\di S(x) \notag \\
	&= c_0^3\lambda_0^2\abs{\partial\brum} + c_0c_1^2\rho^{-2}\lambda_1(\lambda_1+2\lambda_0) \int_{\partial\brum} x_1^2 \,\di S(x) \notag \\
	&= c_0^3\lambda_0^2\abs{\partial\brum} + c_0c_1^2\rho^{-2}\lambda_1(\lambda_1+2\lambda_0)\abs{\partial\brum}d^{-1}, \label{eq:otherbndnorm}
	\end{align}
    where the odd powers of $x \cdot \as$ under the integral vanish due to symmetry, and the third equality follows by renaming the line spanned by $a$ as the first coordinate axis. However, evaluating the integral of $x_1^2$ over $\partial \brum$ to arrive at \eqref{eq:otherbndnorm} requires some extra calculations:

    We recall first a few facts about the gamma function, namely $\Gamma(m+1)=m!$ for $m\in\rum{N}_0$, $\abs{\partial\brum} = 2\pi^{d/2}/\Gamma(\frac{d}{2})$, $\Gamma(\frac{1}{2}) = \pi^{1/2}$, $\Gamma(\frac{3}{2}) = \frac{1}{2}\pi^{1/2}$, and Legendre's duplication formula for $z\in\rum{C}$ with $\redel(z)>0$,
	\begin{equation}
	\Gamma(z+\tfrac{1}{2}) = 2^{1-2z}\pi^{1/2}\frac{\Gamma(2z)}{\Gamma(z)}. \label{eq:legendreformula}
	\end{equation}
	Applying \eqref{eq:legendreformula} twice, we get
	\begin{equation}
	\frac{\Gamma(z)}{2 \Gamma(z+1)} = \frac{\Gamma(2z)}{\Gamma(2z+1)}. \label{eq:legendaformula2}
	\end{equation}
	Using the appropriate formula from \cite{Folland_2001} together with \eqref{eq:legendaformula2} finally gives
	\begin{equation*}
	\int_{\partial\brum} x_1^2\,\di S(x) = \frac{2\Gamma(\tfrac{3}{2})\Gamma(\tfrac{1}{2})^{d-1}}{\Gamma(\tfrac{d}{2}+1)} = \frac{\pi^{d/2}}{\Gamma(\tfrac{d}{2}+1)} = \abs{\partial\brum}\frac{\Gamma(\tfrac{d}{2})}{2\Gamma(\tfrac{d}{2}+1)} = \abs{\partial\brum} \frac{\Gamma(d)}{\Gamma(d+1)} = \abs{\partial\brum}d^{-1},
	\end{equation*}
        which completes the proof of \eqref{eq:otherbndnorm}.
        
	Now we are finally ready to derive the upper bound in \eqref{eq:depth_bounds}. Since $\lambda_0 = \norm{\Lambda_{0,r}-\Lambda_1}_{\mathscr{L}(L^2(\partial\brum))}$ by Remark~\ref{remark:eigval}, the formulas \eqref{eq:normequality}--\eqref{eq:otherbndnorm} yield
	\begin{align*}
	\frac{\norm{\Lambda_{0,r}-\Lambda_1}_{\mathscr{L}(L^2(\partial\brum))}^2}{\norm{\Lambda_{C,R}-\Lambda_1}_{\mathscr{L}(L^2(\partial\brum))}^2}	&\leq \frac{\norm{g_a^{-1}}_{L^2(\partial\brum)}^2\lambda_0^2}{\norm{(c_0+c_1 f_1)^{1/2}(c_0\lambda_0 + c_1\lambda_1 f_1) }_{L^2(\partial\brum)}^2} \\
	&= \frac{\lambda_0^2}{c_0^2\lambda_0^2 + c_1^2\rho^{-2}\lambda_1(\lambda_1+2\lambda_0)d^{-1}} \\
	&= \frac{d}{c_0^2 d +c_1^2\rho^{-2}\frac{\lambda_1}{\lambda_0}\left(\frac{\lambda_1}{\lambda_0}+2\right)}.
	\end{align*}
	Taking the square root and inserting $c_0 = (1+\rho^2)/(1-\rho^2)$ and $c_1 =-2\rho^2/(1-\rho^2)$, we finally arrive at
	\begin{equation*}
	\frac{\norm{\Lambda_{0,r}-\Lambda_1}_{\mathscr{L}(L^2(\partial\brum))}}{\norm{\Lambda_{C,R}-\Lambda_1}_{\mathscr{L}(L^2(\partial\brum))}}	\leq \left[\frac{(1-\rho^2)^2 d}{(1+\rho^2)^2 d +4\rho^2\frac{\lambda_1}{\lambda_0}\left(\frac{\lambda_1}{\lambda_0}+2\right)}\right]^{1/2} \leq \frac{1-\rho^2}{1+\rho^2}.
	\end{equation*}
	
	What remains to be shown is that the derived bounds are optimal in the sense of \eqref{eq:depth_bounds_opti}; in fact, we will demonstrate that the upper bound is reached when $r \to 0^+$ and the lower bound at the opposite extreme $r \to 1^-$. To begin with, note that the ratio $\lambda_n/\lambda_0$, $n\geq 1$, can be written as a function of  $r\in(0,1)$ and $d \in \mathbb{N} \setminus \{ 1 \}$ as
	\begin{equation*}
	\frac{\lambda_n}{\lambda_0} = \begin{dcases}
	\frac{-2nr^{2n}\log(r)}{1-r^{2n}}, &  \quad d = 2, \\
	\frac{(2n+d-2)r^{2n}(1-r^{d-2})}{(d-2)(1-r^{2n+d-2})}, & \quad d \geq 3.
	\end{dcases} 
	\end{equation*}
	In particular, $\lambda_n/\lambda_0$ is an increasing function of $r$ with $\lim_{r\to 0^+} \lambda_n/\lambda_0 = 0$ and $\lim_{r\to 1^-}\lambda_n/\lambda_0 = 1$ for $n\geq 1$. The spectral decomposition of $\Lambda_{0,r}-\Lambda_1$ thus reveals
	\begin{equation*}
	\frac{\Lambda_{0,r}-\Lambda_1}{\norm{\Lambda_{0,r}-\Lambda_1}_{\mathscr{L}(L^2(\partial\brum))}} = P + T_r,
	\end{equation*}
	where $P$ is the orthogonal projection in $\mathscr{L}(L^2(\partial\brum))$ onto constant functions, and $T_r$ is a positive semi-definite operator with $\norm{T_r}_{\mathscr{L}(L^2(\partial\brum))} = \frac{\lambda_1}{\lambda_0} = o(r)$. To be more precise,
	\begin{equation*}
	T_r = \sum_{n=1}^\infty\sum_{j=1}^{\alpha_{n,d}} \frac{\lambda_n}{\lambda_0}\inner{\,\cdot\,,f_{n,j}}_{L^2(\partial\brum)}f_{n,j}.
	\end{equation*}
	Moreover, both $G_a^{-1}P G_a^{-1}$ and $G_a^{-1}T_rG_a^{-1}$ are positive semi-definite, and clearly it also holds $\norm{G_a^{-1}T_rG_a^{-1}}_{\mathscr{L}(L^2(\partial\brum))} = o(r)$. Hence, \eqref{eq:normequality} leads to
	\begin{align*}
	\sup_{r\in(0,1)}\frac{\norm{\Lambda_{0,r}-\Lambda_1}_{\mathscr{L}(L^2(\partial\brum))}}{\norm{\Lambda_{C,R}-\Lambda_1}_{\mathscr{L}(L^2(\partial\brum))}} \hspace{-3cm}&\\
	&= \sup_{r\in(0,1)} \norm{G_a^{-1}PG_a^{-1} + G_a^{-1}T_rG_a^{-1}}_{\mathscr{L}(L^2(\partial\brum))}^{-1} \\
	&= \sup_{r\in(0,1)} \inf_{\norm{f}_{L^2(\partial\brum)}= 1} \left( \inner{G_a^{-1}PG_a^{-1}f,f}_{L^2(\partial\brum)} + \inner{G_a^{-1}T_rG_a^{-1}f,f}_{L^2(\partial\brum)} \right)^{-1} \\
	&= \inf_{\norm{f}_{L^2(\partial\brum)}= 1}  \inner{G_a^{-1}PG_a^{-1}f,f}_{L^2(\partial\brum)}^{-1} \\
	&= \norm{G_a^{-1}PG_a^{-1}}_{\mathscr{L}(L^2(\partial\brum))}^{-1}.
	\end{align*}
	Because of \eqref{eq:ga_m2}, $G_a^{-1}PG_a^{-1}g_a^{-1} = c_0g_a^{-1}$,~i.e.,~$c_0$ is the only nonzero eigenvalue of the self-adjoint rank one operator $G_a^{-1}PG_a^{-1}$. In particular, $c_0$ equals the operator norm of $G_a^{-1}PG_a^{-1}$, which gives
	\begin{equation*}
	\sup_{r\in(0,1)}\frac{\norm{\Lambda_{0,r}-\Lambda_1}_{\mathscr{L}(L^2(\partial\brum))}}{\norm{\Lambda_{C,R}-\Lambda_1}_{\mathscr{L}(L^2(\partial\brum))}} = c_0^{-1} = \frac{1-\rho^2}{1+\rho^2},
	\end{equation*}
	proving one half of \eqref{eq:depth_bounds_opti}.
        
	The second half of \eqref{eq:depth_bounds_opti} follows from a similar line of reasoning. Since $\frac{\lambda_n}{\lambda_0} \nearrow 1$ when $r\to 1^-$, it follows that $T_r \to \ident - P$ in the strong operator topology as $r\to 1^-$ by virtue of dominated convergence. Moreover, $\inner{T_r f,f}_{L^2(\partial\brum)}$ is obviously nondecreasing with respect to $r\in(0,1)$ for each $f\in L^2(\partial\brum)$. This gives
	\begin{align*}
	\inf_{r\in(0,1)}\frac{\norm{\Lambda_{0,r}-\Lambda_1}_{\mathscr{L}(L^2(\partial\brum))}}{\norm{\Lambda_{C,R}-\Lambda_1}_{\mathscr{L}(L^2(\partial\brum))}} \hspace{-3cm}&\\
	&= \inf_{r\in(0,1)} \inf_{\norm{f}_{L^2(\partial\brum)}= 1} \left( \inner{G_a^{-1}PG_a^{-1}f,f}_{L^2(\partial\brum)} + \inner{G_a^{-1}T_rG_a^{-1}f,f}_{L^2(\partial\brum)} \right)^{-1} \\
	&= \inf_{r\in(0,1)} \inf_{\norm{f}_{L^2(\partial\brum)}= 1} \left( \inner{PG_a^{-1}f,G_a^{-1}f}_{L^2(\partial\brum)} + \inner{T_rG_a^{-1}f,G_a^{-1}f}_{L^2(\partial\brum)} \right)^{-1} \\
	&= \inf_{\norm{f}_{L^2(\partial\brum)}= 1} \left( \inner{PG_a^{-1}f,G_a^{-1}f}_{L^2(\partial\brum)} + \inner{(\ident-P)G_a^{-1}f,G_a^{-1}f}_{L^2(\partial\brum)} \right)^{-1} \\
	&= \inf_{\norm{f}_{L^2(\partial\brum)}= 1} \inner{G_a^{-2}f,f}_{L^2(\partial\brum)}^{-1} = \norm{G_a^{-1}}_{\mathscr{L}(L^2(\partial\brum))}^{-2} = \frac{1-\rho}{1+\rho},
	\end{align*}
	where the last equality follows from Lemma~\ref{lemma:opnorms}(iii). This completes the proof.
\end{proof}

The $r$-dependent upper bound in \eqref{eq:depth_bounds} satisfies 
\begin{equation*}
C_d(\rho) := \left(\frac{(1-\rho^2)^2 d}{(1+\rho^2)^2 d +12\rho^2}\right)^{1/2} \leq \left(\frac{(1-\rho^2)^2 d}{(1+\rho^2)^2 d +4\rho^2\frac{\lambda_1}{\lambda_0}\left(\frac{\lambda_1}{\lambda_0}+2\right)}\right)^{1/2},
\end{equation*} 
where $C_d(\rho)$ annotates the \emph{least upper bound}, attained when $r\to 1^-$, for a given dimension $d$. It is evident that $C_d(\rho)\to (1-\rho^2)/(1+\rho^2)$ as $d\to \infty$,~i.e.,~the effect that $r \in (0,1)$ has on the $r$-dependent upper bound diminishes as the dimension grows. Figure~\ref{fig:boundplot} compares the least upper bound $C_d(\rho)$ with the $r$-independent lower and upper bounds from \eqref{eq:depth_bounds} for $d=2, \dots, 15$. 

\begin{figure}[htb]
	\centering
	\includegraphics[width=.6\textwidth]{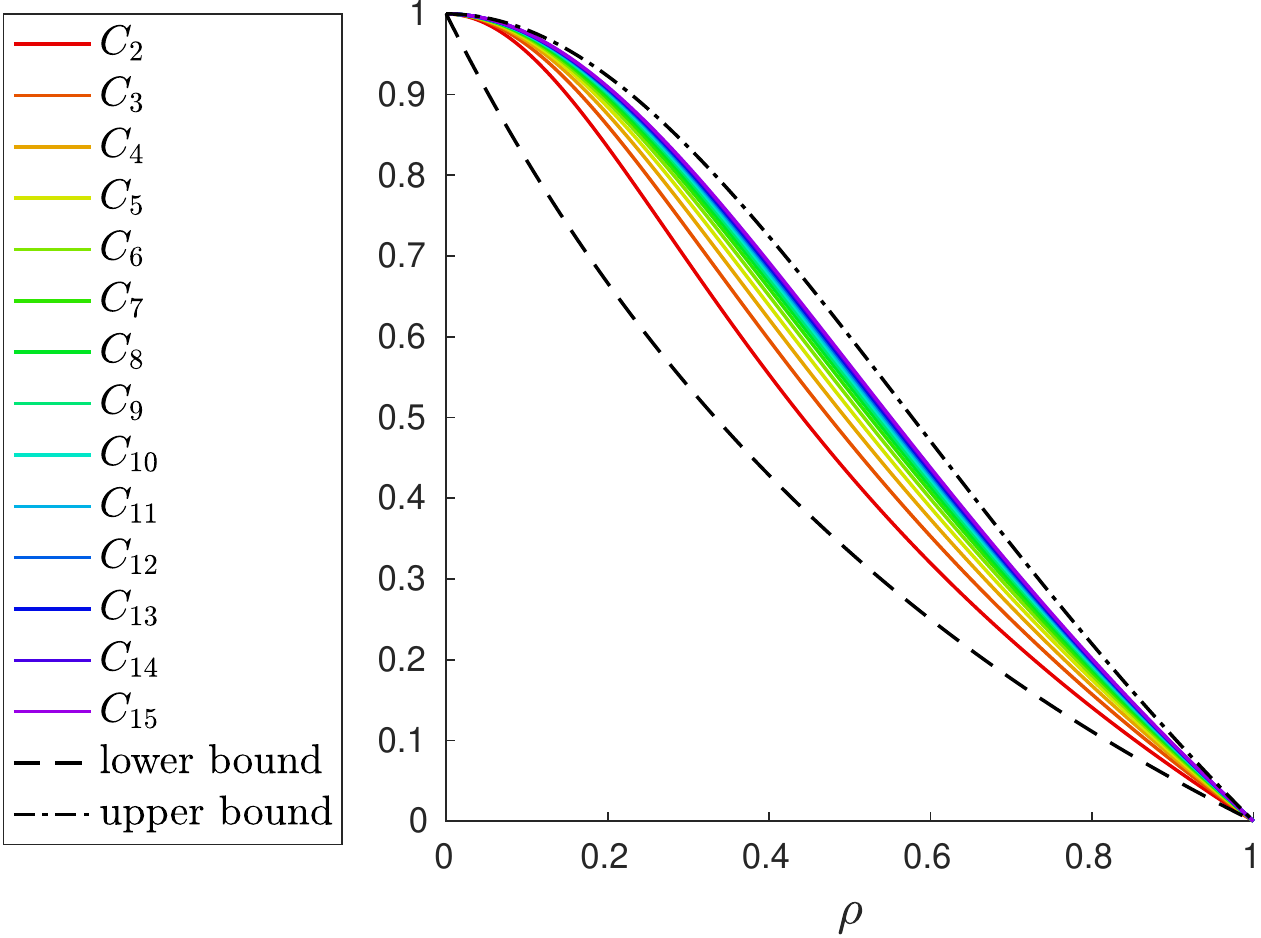}
	\caption{The lower bound $\tfrac{1-\rho}{1+\rho}$ and the upper bound $\tfrac{1-\rho^2}{1+\rho^2}$ from \eqref{eq:depth_bounds} compared with the least upper bound $C_d(\rho)$ for $d=2,\dots,15$ as functions of $\rho \in (0,1)$.} \label{fig:boundplot}
\end{figure}

\begin{remark}
By comparing \eqref{eq:depth_bounds} to the numerical studies in \cite{Garde_2017}, it is observed that the upper bound $\smash{\frac{1-\rho^2}{1+\rho^2}}$ in Theorem~\ref{thm:L2bnds} also seems to be tight for inclusions of finite conductivity in two spatial dimensions.
\end{remark}

\begin{remark}
	By choosing $f = \ga$ instead of $f = g_a^{-1}$ in \eqref{eq:bnd1}, one may exploit the fact that the first eigenfunction of $\Lambda_{0,r}-\Lambda_1$ is constant to deduce another lower bound for the operator norm of $G_a^{-1}(\Lambda_{0,r}-\Lambda_1)G_a^{-1}$. Using the slice integration formula in \cite[Corollary~A.5]{Axler_2001}, this leads to a different (i.e.~worse) upper bound (cf.~\eqref{eq:depth_bounds})
	\begin{equation}
		\frac{\norm{\ga}_{\mathscr{L}(L^2(\partial\brum))}}{\norm{\ga^{-1}}_{\mathscr{L}(L^2(\partial\brum))}} = \frac{1-\rho^2}{\sqrt{1+\rho^2}}\left(\frac{(d-1)V_{d-1}}{dV_d}\int_{-1}^1 \frac{(1-y^2)^{(d-3)/2}}{1+\rho^2-2\rho y}\,\di y\right)^{1/2}, \label{eq:worsebnd}
	\end{equation}
	where $V_m$ denotes the volume of the unit ball in $\rum{R}^m$ for $m\in\rum{N}$. The integral in \eqref{eq:worsebnd} allows an explicit expression, which can be found by several applications of the binomial theorem, Ruffini's rule, and the connection between gamma and beta functions. The bound \eqref{eq:worsebnd} improves as $d$ increases --- and tends from above towards the upper bound of Theorem~\ref{thm:L2bnds}. Most notably, for $d=2$ it gives $\smash{(\frac{1-\rho^2}{1+\rho^2})^{1/2}}$ which is the (nonoptimal) upper bound found in \cite[Theorem~3.5]{Garde_2017} for two spatial dimensions.
\end{remark}

\subsection*{Acknowledgments} 

This work was supported by the Academy of Finland (decision 312124) and the Aalto Science Institute (AScI).

\appendix
\section{Comparison with M\"obius transformations in two dimensions} \label{sec:appA}

To ease the comparison of our results and techniques with those in \cite{Garde_2017} for the two-dimensional setting, we briefly analyze the relation between the M\"obius transformations and inversions that map the unit disk onto itself. Since $K_a = \ia$ for $d=2$, this should be enough for convincing the reader that the two approaches are essentially the same in two dimensions.

We identify $\brum$ with the unit disk in the complex plane. In particular, $\as = a/\abs{a}^2 = 1/\overline{a}$ for $a \in \brum\setminus\{0\} \subset \mathbb{C}$. All M\"obius transformations that map the unit disk onto itself are of the form 
\begin{equation*}
M_a(x) := \frac{x-a}{\overline{a}x-1} = \frac{\as(x-a)\overline{(x-\as)}}{\abs{x-\as}^2} = (\as - a)\frac{\overline{ \overline{\as}(x-\as) } }{\abs{x-\as}^2} + \as, \qquad a \in \brum\setminus\{0\},
\end{equation*}
up to rotations. We can also rewrite $I_a$ in a similar form,
\begin{equation*}
I_a(x) = (\abs{\as}^2-1)\frac{x-\as}{\abs{x-\as}^2}+\as = (\as - a)\frac{\overline{\as}(x-\as)}{\abs{x-\as}^2} + \as.
\end{equation*}
In particular, for $\rho = \abs{a}$ it holds $I_\rho(x) = \overline{M_\rho(x)}$. On the other hand, if $a = \rho \e^{\I\zeta}$ for some $\zeta\in\rum{R}$, it is easy to verify that
\begin{equation*}
M_{a}(x) = \e^{\I\zeta}M_\rho (\e^{-\I\zeta}x) \qquad \text{and} \qquad
I_a(x) = \e^{\I\zeta} I_\rho(\e^{-\I\zeta}x), \qquad x \in  \brum.
\end{equation*}
These two observations immediately lead to the identity $\e^{-\I\zeta}I_a = \overline{e^{-\I\zeta}M_a}$. Since complex conjugation corresponds to reflection with respect to the real axis, we have proved the identity $I_a = {\rm Ref}_a \circ M_a$, where ${\rm Ref}_a$ denotes the reflection with respect to the line spanned by $a$.

Obviously,
\begin{equation*}
\abs{I_a(x)} = \abs{M_a(x)} = \frac{\abs{x-a}}{\abs{\overline{a}x-1}} = \frac{\abs{x-a}}{\abs{x-\as}}\abs{\as} =: r_{x,a}, \qquad x \in  \brum,
\end{equation*}
and \eqref{eq:eqident} indicates
\begin{equation*}
\abs{I_a(x)-\as} = \frac{b^2}{\abs{x-\as}} =: \tilde{r}_{x,a}, \qquad x \in  \brum.
\end{equation*}
Due to the symmetry of $I_a(x)$ and $M_a(x)$ about the line $\mspan\{a\}=\mspan\{\as\}$, it is easy to geometrically deduce that $M_a(x)$ and $I_a(x)$ are the two intersections of the circles $S(0,r_{x,a})$ and $S(\as,\tilde{r}_{x,a})$.  See Figure~\ref{fig:geofig} for a visualization of this geometric interpretation. In particular, the concept of `depth', characterized through $\rho = \abs{a}$, is equivalent for the two transformations.

\begin{figure}[htb]
	\centering
	\includegraphics[width=.75\textwidth]{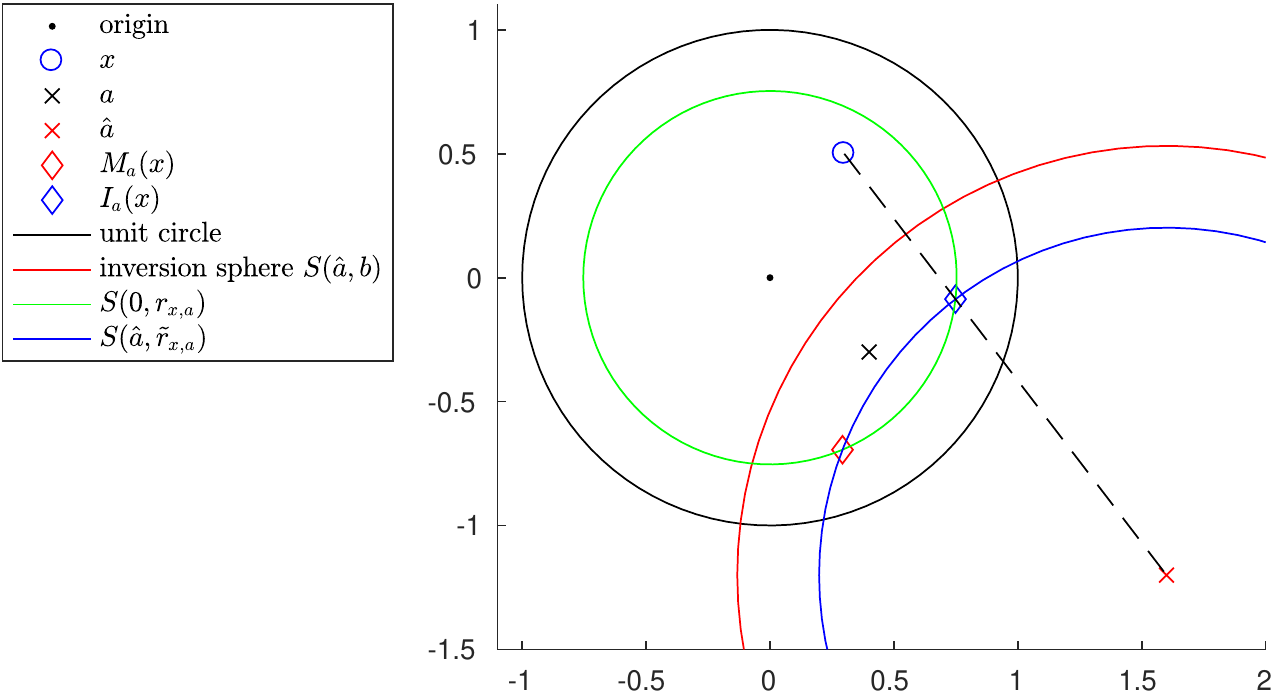}
	\caption{Geometric interpretations of $I_a(x)$ and $M_a(x)$ in two dimensions.} \label{fig:geofig}
\end{figure}

\section{A representation formula for a Kelvin-transformed DN map} \label{sec:appB}

In this appendix, we present a diagonalization of $\Lambda_{C,R}-\Lambda_1$ that we consider interesting in its own right, even though it is not needed when proving the main result of this work. Before proceeding further, we recommend reviewing the summary on spherical harmonics in Section~\ref{sec:EIT} and the definition of the weighted $L^2$-spaces on $\partial \brum$ introduced in Section~\ref{sec:bounds}. 

Observe that $\{\phi_{n,j}\} := \{\Ka f_{n,j}\}$ is an orthonormal basis for $L^2_{a,1}(\partial\brum)$ and $\{\psi_{n,j}\} := \{G_a^2\Ka f_{n,j}\}$ is such for $L^2_{a,-1}(\partial\brum)$ due to Lemma~\ref{lemma:opnorms}(ii) and $K_a$ being an involution. In particular, both of these are (non-orthonormal) bases for the standard space $L^2(\partial \brum)$. This leads to the following representation of $\Lambda_{C,R}-\Lambda_1$.

\begin{proposition} \label{prop:lamdiag}
	Assume $B(C,R) = I_a(B(0,r))$ for $a\in\brum\setminus\{0\}$ and $r\in(0,1)$. Let $\{\lambda_{n}\}_{n\in\rum{N}_0}$ denote the set of eigenvalues for $\Lambda_{0,r}-\Lambda_1$, cf.~Proposition~\ref{prop:eigvals}. Then $\diag(\{\lambda_n\})$, with each $\lambda_n$ repeated according to its multiplicity $\alpha_{n,d}$, is a matrix representation for $\Lambda_{C,R}-\Lambda_1$ with respect to the bases $\{\phi_{n,j}\}$ and $\{\psi_{n,j}\}$ of $L^2(\partial \brum)$, that is,
	\begin{equation*}
	\inner{(\Lambda_{C,R}-\Lambda_1)\phi_{m,j'},\psi_{n,j}}_{a,-1} = \lambda_m\delta_{m,n}\delta_{j',j}. 
	\end{equation*}
	In particular,
	\begin{equation}
	(\Lambda_{C,R}-\Lambda_1)f = \sum_{n=0}^\infty\sum_{j=1}^{\alpha_{n,d}} \lambda_n \inner{f,\psi_{n,j}}_{L^2(\partial\brum)}\psi_{n,j} \label{eq:sumformula}
	\end{equation}
	for any $f\in L^2(\partial\brum)$.
\end{proposition}
\begin{proof}
	As $K_a$ is an involution, $(\lambda_m,\phi_{m,j'})$ is an eigenpair of $\Ka (\Lambda_{0,r}-\Lambda_1)\Ka$. Hence, due to Theorem~\ref{thm:DNkelvin} and Lemma~\ref{lemma:opnorms}(ii),
	\begin{align*}
	\inner{(\Lambda_{C,R}-\Lambda_1)\phi_{m,j'},\psi_{n,j}}_{a,-1} &= \inner{G_a\Ka(\Lambda_{0,r}-\Lambda_1)\Ka\phi_{m,j'}, G_a \Ka f_{n,j}}_{L^2(\partial\brum)} \\[1mm]
	&= \lambda_m \inner{G_a^2 \Ka f_{m,j'},\Ka f_{n,j}}_{L^2(\partial\brum)} \\[1mm]
	&= \lambda_m \inner{f_{m,j'},f_{n,j}}_{L^2(\partial\brum)},
	\end{align*}	
	which proves the first part of the claim as $\{\psi_{n,j}\}$ is an orthonormal basis for $L^2_{a,-1}(\partial\brum)$ and $\{f_{n,j}\}$ is such for $L^2(\partial \brum)$.
	
	To show \eqref{eq:sumformula}, we simply write up expansions for $(\Lambda_{C,R}-\Lambda_1)f$ and $f$ in terms of $\{\psi_{n,j}\}$ and $\{\phi_{m,j'}\}$, respectively, and apply the above result,
	\begin{align*}
	(\Lambda_{C,R}-\Lambda_1)f &= \sum_{n=0}^\infty\sum_{j=1}^{\alpha_{n,d}} \inner{(\Lambda_{C,R}-\Lambda_1)f,\psi_{n,j}}_{a,-1}\psi_{n,j} \\
	&= \sum_{n=0}^\infty\sum_{m=0}^\infty \sum_{j=1}^{\alpha_{n,d}}\sum_{j'=1}^{\alpha_{m,d}} \inner{(\Lambda_{C,R}-\Lambda_1)\phi_{m,j'},\psi_{n,j}}_{a,-1}\inner{f,\phi_{m,j'}}_{a,1}\psi_{n,j} \\
	&= \sum_{n=0}^\infty\sum_{j=1}^{\alpha_{n,d}} \lambda_n \inner{f,\phi_{n,j}}_{a,1}\psi_{n,j},
	\end{align*}
	which completes the proof as $\inner{f,\phi_{n,j}}_{a,1} = \inner{f,\psi_{n,j}}_{L^2(\partial\brum)}$.
\end{proof}

\begin{remark}
	It is worth noting that \eqref{eq:sumformula} is not a spectral decomposition of $\Lambda_{C,R}-\Lambda_1$ since $\{\psi_{n,j}\}$ in Proposition~\ref{prop:lamdiag} is not an {\em orthonormal} basis for $L^2(\partial\brum)$. In fact, as seen in the proof, it provides a spectral decomposition of $G_a^{-2}(\Lambda_{C,R}-\Lambda_1)$ with eigenpairs $(\lambda_n,\phi_{n,j})$.
\end{remark}

\bibliographystyle{plain}
\bibliography{minbib}

\begin{thebibliography}{10}

\bibitem{sobolev}
R.~A. Adams and J.~J.~F. Fournier.
\newblock {\em Sobolev spaces}, volume 140 of {\em Pure and Applied Mathematics
  (Amsterdam)}.
\newblock Elsevier/Academic Press, Amsterdam, second edition, 2003.

\bibitem{Alessandrini1988}
G.~Alessandrini.
\newblock Stable determination of conductivity by boundary measurements.
\newblock {\em Appl. Anal.}, 27:153--172, 1988.

\bibitem{Alessandrini2001}
G.~Alessandrini and R.~Gaburro.
\newblock Determining conductivity with special anisotropy by boundary
  measurements.
\newblock {\em SIAM J. Math. Anal.}, 33(1):153--171, 2001.

\bibitem{Alessandrini2009}
G.~Alessandrini and R.~Gaburro.
\newblock The local {C}alder\'on problem and the determination at the boundary
  of the conductivity.
\newblock {\em Comm. PDE}, 34(7-9):918--936, 2009.

\bibitem{Alessandrini_2017}
G.~Alessandrini and A.~Scapin.
\newblock Depth dependent resolution in electrical impedance tomography.
\newblock {\em J. Inverse Ill-Posed Probl.}, 25(3), 2017.

\bibitem{Ammari_2013}
H.~Ammari, J.~Garnier, and K.~S{\o}lna.
\newblock Partial data resolving power of conductivity imaging from boundary
  measurements.
\newblock {\em {SIAM} J. Math. Anal.}, 45(3):1704--1722, 2013.

\bibitem{Ammari2007a}
H.~Ammari and H.~Kang.
\newblock {\em Polarization and Moment Tensors: with Applications to Inverse
  Problems and Effective Medium Theory}, volume 162 of {\em Applied
  Mathematical Sciences}.
\newblock Springer-Verlag, New York, 2007.

\bibitem{Armitage_2001}
D.~H. Armitage and S.~J. Gardiner.
\newblock {\em Classical Potential Theory}.
\newblock Springer London, 2001.

\bibitem{Axler_2001}
S.~Axler, P.~Bourdon, and W.~Ramey.
\newblock {\em Harmonic Function Theory}.
\newblock Springer New York, 2001.

\bibitem{Bogdan_2006}
K.~Bogdan and T.~{\.{Z}}ak.
\newblock On {K}elvin {T}ransformation.
\newblock {\em J. Theor. Probab.}, 19(1):89--120, 2006.

\bibitem{Borman_2009}
D.~Borman, D.~B. Ingham, B.~T. Johansson, and D.~Lesnic.
\newblock The method of fundamental solutions for detection of cavities in
  {EIT}.
\newblock {\em J. Integral Equ. Appl.}, 21(3):383--406, 2009.

\bibitem{Brown2001a}
R.~M. Brown.
\newblock Recovering the conductivity at the boundary from the {D}irichlet to
  {N}eumann map: a pointwise result.
\newblock {\em J. Inverse Ill-Posed Probl.}, 9(6):567--574, 2001.

\bibitem{Cheney1992}
M.~Cheney and D.~Isaacson.
\newblock Distinguishability in impedance imaging.
\newblock {\em {IEEE} T. Bio-Med. Eng.}, 39(8):852--860, 1992.

\bibitem{Dautray1988}
R.~Dautray and J.-L. Lions.
\newblock {\em Mathematical Analysis and Numerical Methods for Science and
  Technology}, volume~2.
\newblock Springer, 1988.

\bibitem{Efthimiou_2014}
C.~Efthimiou and C.~Frye.
\newblock {\em Spherical Harmonics in $p$ Dimensions}.
\newblock World Scientific, 2014.

\bibitem{Erhard2003}
K.~Erhard and R.~Potthast.
\newblock The point source method for reconstructing an inclusion from boundary
  measurements in electrical impedance tomography and acoustic scattering.
\newblock {\em Inverse Problems}, 19:1139--1157, 2003.

\bibitem{Folland_2001}
G.~B. Folland.
\newblock How to integrate a polynomial over a sphere.
\newblock {\em Amer. Math. Monthly}, 108(5):446--448, 2001.

\bibitem{Garde_2017}
H.~Garde and K.~Knudsen.
\newblock Distinguishability revisited: Depth dependent bounds on
  reconstruction quality in electrical impedance tomography.
\newblock {\em {SIAM} J. Appl. Math.}, 77(2):697--720, 2017.

\bibitem{Gurarie}
D.~Gurarie.
\newblock {\em Symmetries and {L}aplacians}, volume 174 of {\em North-Holland
  Mathematics Studies}.
\newblock North-Holland Publishing Co., Amsterdam, 1992.
\newblock Introduction to harmonic analysis, group representations and
  applications.

\bibitem{Hanke_2011}
M.~Hanke, L.~Harhanen, N.~Hyv\"onen, and E.~Schweickert.
\newblock Convex source support in three dimensions.
\newblock {\em {BIT} Numer. Math.}, 52(1):45--63, 2011.

\bibitem{HormanderI}
L.~H\"{o}rmander.
\newblock {\em The analysis of linear partial differential operators {I}}.
\newblock Classics in Mathematics. Springer-Verlag, Berlin, 2003.

\bibitem{Isaacson1986}
D.~Isaacson.
\newblock Distinguishability of conductivities by electric current computed
  tomography.
\newblock {\em {IEEE} T. Med. Imaging}, 5(2):91--95, 1986.

\bibitem{Kang2002}
H.~Kang and K.~Yun.
\newblock Boundary determination of conductivities and {R}iemannian metrics via
  local {D}irichlet-to-{N}eumann operator.
\newblock {\em SIAM J. Math. Anal.}, 34:719--735, 2002.

\bibitem{Kelvin}
W.~Thomson~(Lord Kelvin).
\newblock Extraits de deux lettres adress\'ees \`a {M}. {L}iouville.
\newblock {\em J. Math. Pures Appl.}, 12:256--264, 1847.

\bibitem{Kress_2011}
R.~Kress.
\newblock Conformal mapping and impedance tomography.
\newblock {\em J. Phys.: Conf. Ser.}, 290:012009, 2011.

\bibitem{Kress_2012}
R.~Kress.
\newblock Inverse problems and conformal mapping.
\newblock {\em Complex Var. Elliptic}, 57(2-4):301--316, 2012.

\bibitem{Lee1989}
J.~Lee and G.~Uhlmann.
\newblock Determining anisotropic real-analytic conductivities by boundary
  measurements.
\newblock {\em Comm. Pure Appl. Math.}, 42(2):1097--1112, 1989.

\bibitem{Lions1972}
J.-L. Lions and E.~Magenes.
\newblock {\em Non-homogeneous boundary value problems and applications, Vol.
  1}.
\newblock Springer--Verlag, New York--Heidelberg, 1972.

\bibitem{Mandache2001}
N.~Mandache.
\newblock Exponential instability in an inverse problem for the {S}chr\"odinger
  equation.
\newblock {\em Inverse Problems}, 17(5):1435--1444, 2001.

\bibitem{Michalik_2012}
K.~Michalik and M.~Ryznar.
\newblock Kelvin transform for $\alpha$-harmonic functions in regular domains.
\newblock {\em Demonstratio Math.}, 45(2), 2012.

\bibitem{Nagayasu_2009}
S.~Nagayasu, G.~Uhlmann, and J.-N. Wang.
\newblock A depth-dependent stability estimate in electrical impedance
  tomography.
\newblock {\em Inverse Problems}, 25(7):075001, 2009.

\bibitem{Nakamura2001}
G.~Nakamura and K.~Tanuma.
\newblock Direct determination of the derivatives of conductivity at the
  boundary from the localized {D}irichlet to {N}eumann map.
\newblock {\em Comm. Korean Math. Soc.}, 16:415--425, 2001.

\bibitem{Nakamura2001a}
G.~Nakamura and K.~Tanuma.
\newblock Local determination of conductivity at the boundary from the
  {D}irichlet-to-{N}eumann map.
\newblock {\em Inverse Problems}, 17:405--419, 2001.

\bibitem{Nakamura2003}
G.~Nakamura and K.~Tanuma.
\newblock Formulas for reconstructing conductivity and its normal derivative at
  the boundary from the localized {D}irichlet to {N}eumann map.
\newblock In {\em Recent development in theories \& numerics}, pages 192--201.
  World Sci. Publ., River Edge, NJ, 2003.

\bibitem{Seo_2002}
J.~K. Seo, O.~Kwon, and S.~Kim.
\newblock Location search techniques for a grounded conductor.
\newblock {\em {SIAM} J. Appl. Math.}, 62(4):1383--1393, 2002.

\bibitem{Sylvester1988}
J.~Sylvester and G.~Uhlmann.
\newblock Inverse boundary value problems at the boundary --- continuous
  dependence.
\newblock {\em Comm. Pure Appl. Math.}, 41:197--219, 1988.

\bibitem{Wermer_1974}
J.~Wermer.
\newblock {\em Potential Theory}.
\newblock Springer Berlin Heidelberg, 1974.

\end{thebibliography}

\end{document}